\documentclass[a4paper,10pt]{amsart}

\usepackage{inputenc}
\usepackage{color}
\usepackage{graphicx}
\usepackage{times,mathptm}
\usepackage{amsfonts,amscd,amssymb,amsmath}
\usepackage{xypic}
\usepackage{latexsym}
\usepackage{tikz}
\usetikzlibrary{arrows,snakes,backgrounds,decorations.markings}
\usetikzlibrary{patterns,arrows,decorations.pathreplacing}
\usepackage{tikz-cd}
\usepackage{verbatim}
\usepackage{mathtools}
\DeclarePairedDelimiter{\ceil}{\lceil}{\rceil}

\newtheorem{proposition}{Proposition}[section]
\newtheorem{lemma}[proposition]{Lemma}
\newtheorem{corollary}[proposition]{Corollary}
\newtheorem{theorem}[proposition]{Theorem}

\theoremstyle{definition}
\newtheorem{definition}[proposition]{Definition}
\newtheorem{example}[proposition]{Example}

\theoremstyle{remark}
\newtheorem{remark}[proposition]{Remark}

\newcommand{\proplabel}[1]{\label{prop:#1}}
\newcommand{\propref}[1]{Proposition~\ref{prop:#1}}

\newcommand{\lemlabel}[1]{\label{lem:#1}}
\newcommand{\lemref}[1]{Lemma~\ref{lem:#1}}

\newcommand{\thelabel}[1]{\label{the:#1}}
\newcommand{\theref}[1]{Theorem~\ref{the:#1}}

\newcommand{\corlabel}[1]{\label{cor:#1}}
\newcommand{\corref}[1]{Corollary~\ref{cor:#1}}

\newcommand{\remlabel}[1]{\label{rem:#1}}

\newcommand{\exalabel}[1]{\label{ex:#1}}
\newcommand{\exaref}[1]{Example~\ref{ex:#1}}

\newcommand{\seclabel}[1]{\label{sec:#1}}
\newcommand{\secref}[1]{Section~\ref{sec:#1}}

\newcommand\Hom{{\rm Hom}}

\newcommand\ot{\otimes}

\newcommand\ov{\overline}

\newcommand\el{\rm el}

\newcommand\rad{\rm rad}

\newcommand\BS{\rm BS}
\newcommand\GBS{\rm GBS}

\newcommand\Ann{\rm Ann}
\newcommand\Max{\rm Max}
\newcommand\gr{\rm gr}
\newcommand\Spec{\rm Spec}
\newcommand\tdeg{\rm tdeg}

\newcommand\g{\gamma}

\newcommand\hookmapright[1]{\smash{\mathop{\hookrightarrow}\limits^{#1}}}

\newcommand{\veq}{\mathrel{\rotatebox{90}{$=$}}}

\makeindex
\begin{document}

\title{Glider-Brauer-Severi varieties of central simple algebras}
\author[F. Caenepeel]{Frederik Caenepeel}
\address{Shanghai Center for Mathematical Sciences, Fudan University, Shanghai, China}
\email{frederik\textunderscore{}caenepeel@fudan.edu.cn}
\address{Department of Mathematics, University of Antwerp, Antwerp, Belgium}
\email{Frederik.Caenepeel@uantwerpen.be}
\author[F. Van Oystaeyen]{Fred Van Oystaeyen}
\address{Department of Mathematics, University of Antwerp, Antwerp, Belgium}
\email{Fred.Vanoystaeyen@uantwerpen.be}
\subjclass[2010]{16W60, 16W70}
\keywords{Brauer-Severi variety, glider representation}

\begin{abstract}
The glider Brauer-Severy variety $\GBS(A)$ of a central simple algebra $A$ over a field $K$ is introduced as the set of all irreducible left glider ideals in $A$ for some filtration $FA$. For fields we deduce that $\GBS(K)$ equals $R(K) \times \mathbb{Z}$, the product of the Riemann surface of $K$ and $\mathbb{Z}$. For a csa $A$ over $K$ it turns out that $\GBS(A) = \BS(A) \times \GBS(K)$, where $\BS(A)$ denotes the classical Brauer-Severi variety of $A$.
\end{abstract}

\maketitle

\section{Introduction}

Central simpel algebras stem from classical representation theory of finite groups and they were studied in detail in the theory of the Brauer group of a field. Using descent theory for example, there can be found a strong relation to non-Abelian cohomology in some projective linear group. Brauer-Severi varieties are geometric objects associated to a central simple algebra (csa) and these also relate to the non-Abelian cohomology as before, see \cite{Ar}, \cite{Ser}, \cite{Ch}. Brauer-Severi (BS-)varieties found interesting applications in the geometric theory, for example in the Artin-Mumford example of unirational non-rational varieties, and also in some approaches to the Mercurjev-Suslin theorem on the co-torsion of the second $K$-group. Now the Brauer-Severi variety of a csa $A$ over a field $K$ is defined as the variety of irreducible left ideals of $A$, its variety structure coming from an obvious embedding as a closed subset of some Grassmann variety.\\

The authors introduced the notion of a glider representation based on generalized modules with respect to chains of subrings; in the case of glider representations of a finite group this leads to chains of subgroups and then associated chains of group rings in the original total group ring. We studied the new representation theoretic results for glider representations of finite groups in \cite{CVo1}, \cite{CVo2}, \cite{CVo4}. We defined irreducible glider representations but the definition is slightly more elaborate than the one of irreducible modules. Nevertheless, they work well in the representation theory and so the question prompts itself whether the irreducible subgliders of A itself define something like a $\BS$-variety? The gliders in $A$ are defined with respect to a filtration $FA$ having a ring of filtration-degree zero, $B$ say, inducing a filtration on $K$ with ring of filtration degree zero, $R$ say. We work in the situation where $KB=A$ and usually we assume $R$ is a Noetherian ring, in fact in future work we would like to deal with Noetherian integrally closed domains, then $B$ is an order over $R$ in $A$. The definition of the $\GBS(A)$, the glider Brauer-Severi variety, in terms of irreducible subgliders of $A$ with respect to some filtration $FA$ (with some extra properties usually), does indeed lead to some geometric structure. A first observation is that $\GBS(K)$ is not trivial, unlike the $\BS(K)$, so we study the $\GBS$ of a field first and, perhaps a surprise, we find it is the product of the Riemann surface $R(K)$ of $K$ with the integers $\mathbb{Z}$. The latter is the effect of some shift which is always possible on gliders, we can get rid of this factor $\mathbb{Z}$ by introducing glider classes under shifts, but we do not do that here. The Riemann surface of a field is the set of all discrete valuation rings of the field and it may be given a nice topology of Zariski-type and viewed as some geometric object but not a variety. One of the telling results of the theory about $\GBS(A)$ is that the relative $\GBS$ of $A$ introduced in \secref{GBScsa} turns out to be $\BS(A) \times \GBS(K)$, so the product of the $\BS(A)$ and the Riemann surface $R(K)$ with a further factor $\mathbb{Z}$. In a sense this thus yields some geometric structure in terms of the $\BS$ of $A$ and the Riemann surface of $K$ which is not a variety but still nicely described. The theory of orders and maximal orders enters the glider theory naturally here, for example, we also have a version of the Brandt groupoid appearing. We could go to rings $K$ instead of fields and start from Azumaya algebras and $\BS$-schemes, but this is left for work in progress.\\

The structure of the paper is as follows: In \secref{pre} we recall the necessary notions about glider representations and show some results about filtrations on fields and central simple algebras. For example, in \propref{estep} we make a connection between irreducible glider representations and the filtration being strong. \secref{GBSfield} starts with the definition of $\GBS(A)$, the glider Brauer-Severi variety of a csa $A$. However, the rest of the section is entirely devoted to the calculation of the $\GBS$ of a field $K$. We conclude with \theref{GBSfield}, which shows how the Riemann surface $R(K)$ enters the game. Subsequently, in \secref{GBScsa} we deal with the $\GBS$ of central simple algebras $A$ over a field $K$. In fact, we introduce the relative $\GBS^K(A)$, in which we restrict to filtrations $FA$ that induce separated, exhaustive and unbounded filtrations $FK$. The main result of this section is \corref{GBScsa} which shows that $\GBS^K(A)$ equals $\BS(A) \times \GBS(K)$. In \secref{tensor} we define a tensor product with a field extension $L/K$, which in the case of a strong filtration on a csa $A$ over $K$ gives rise to a map $\GBS_F(A) \to \GBS_f(A \ot_K L)$. These observations then allow to deduce that the relative glider Brauer-Severi variety $\GBS^K(A)$ is a twist of the relative glider Brauer-Severi variety of a matrix algebra $M_n(L)$ for a suitable field extension $L/K$. Throughout these sections, we indicate links with the theory of (maximal) orders. This then inspires the construction of the Brandt groupoid of normal glider ideals in a csa $A$. This is carried out in \secref{Brandt}. Finally, since glider representations can be defined for $\Gamma$-filtrations with $\Gamma$ any totally ordered group, we include a final section in which we work with $\Gamma = \mathbb{Z}^2$ with lexicographical order. We establish \theref{GBS2} which shows that all rank 2 valuation rings in a field $K$ enter the scene. 

\section{Some results on separated, exhaustive filtrations on central simple algebras}\seclabel{pre}

In this section $K$ is a field and $A$ a central simple $K$-algebra. We recall that a filtration $FA$ on $A$ is defined by an ascending chain $\ldots \subset F_{-1}A \subset F_0A \subset F_1A \subset \ldots \subset A$ of additive subgroups of $A$, such that $F_nAF_mA \subset F_{n+m}A$ for all $n,m \in \mathbb{Z}$. In particular, $B = F_0A$ is a subring of $A$. We call $FA$ separated, resp. exhaustive if $\bigcap_n F_nA = 0$, resp. $\bigcup_n F_nA = A$. We always assume that the filtrations considered throughout the text are separated, exhaustive and not bounded. The latter means that all $F_nA \neq 0$ and there exists no $m \in \mathbb{Z}$ such that $F_mA = A$. We call the filtration $FA$ strong if $F_nAF_{-n}A = F_0A = F_{-n}AF_nA$ for all $n \geq 0$. In fact, $FA$ is strong if and only if $F_{1}AF_{-1}A = F_0A = F_{-1}AF_1$.\\

When fixing such a filtration $FA$, one can consider left (or right) $FA$-glider representations, which are given by left $F_0A = B$-modules $M$ embedded in a left $A$-module $\Omega$ together with a chain of descending left $B$-modules
$$M \supset M_1 \supset M_2 \supset \ldots$$
such that for all $i \leq j,~F_iAM_j \subset M_{j-i}$, where the action is the action induced by the $A$-action on $\Omega$, see \cite{CVo1} for the exact definition. In \cite{NVo1} the notion of trivial subgliders was introduced and it was later refined in \cite{CVo1}. We recall that a subglider $N$ of $M$ is said to be trivial if either:\\

$T_1$. There is an $n \in \mathbb{N}$ such that $N_n = B(N)$ but $M_n \neq B(M)$.\\
$T_2$. There is an $n \in \mathbb{N}$ such that $N_n = 0$ but $M_n \neq 0$.\\
$T_3$. There exists a monotone increasing map $\alpha: \mathbb{N} \to \mathbb{N}$ such that $N_n = M_{\alpha(n)}$ and $\alpha(n) - m \geq \alpha(n - m)$ for all $m \leq n$ in $\mathbb{N}$.\\

The monotone increasing map $\alpha$ for type $T_3$ is in fact strict monotone. Indeed, $1 \leq n+1$ implies that $\alpha(n+1) - 1 \geq \alpha(n)$. We also recall the notion of essential length of a glider. If there exists $d \geq 0$ such that $M_d \supsetneq M_{d+1}$, but $M_{d+1} = M_{d+n}$ for all $n > 0$, we say that $M$ is of finite essential length $d$. If such a $d$ does not exist, we say that $M$ is of infinite essential length.\\

As a particular example of a left (and right) $FA$-glider, we have the negative part of the filtration:
$$ F^{-}B: \quad B = F_0A \supset F_{-1}A \supset F_{-2}A \supset \ldots$$
Let us investigate when this glider is in fact irreducible.\\

For ease of notation, we will write $F_n$ rather than $F_nA$. First of all, if $F_0 = F_{-1}$, then for all $n > 0$ we would have $F_n = F_nF_0 = F_nF_{-1} \subset F_{n-1}$. It follows that $F^+A$ is the trivial chain and in order to be exhaustive $F_0 = A$ and then $FA$ is the just the trivial chain, whence is not separated. Thus $F_{-1} \subsetneq F_0$. Also, there are no idempotent elements in the negative part. Indeed, suppose on the contrary that $F_{-n}F_{-n} = F_{-n}$, then for all $k > 0$ it holds $F_{-n} = F_{-n}^k \subset F_{-nk}$.  This would entail that $F_{-n}$ is in the core of the filtration. 
\begin{remark} 
By convention, we know that the left glider ideal $F^-B$ is of infinite essential length. If we would drop the left boundedness condition then $F^{-}B$ being of finite essential length, say $d$, would entail that $FA$ is a positive filtration. Indeed, it would follow that $F_{-d} \supsetneq F_{-d-1} =0$ because the filtration is separated. By irreducibility of $F^-B$, the subglider
$$F_dF_{-d} \supset F_{d-1}F_{-d} \supset F_{1}F_{-d} \supset F_{-d} \supsetneq 0 \supset \ldots$$
must be of type $T_3$ and hence $F_0 = F_{d}F_{-d}$. It follows that $F_{-1}  = F_0F_{-1} = F_dF_{-d}F_{-1} \subset F_dF_{-d-1} = 0$, whence $FA$ is indeed a positive filtration.
\end{remark}

The filtration $FA$ induces on $K$ a filtration $FK$ defined by $F_nK = F_n \cap K$. The filtration is obviously also separated and exhaustive. The subring $F_0K$ must be proper, for otherwise all $F_nA$ are $K$-vector spaces and since $A$ is finite dimensional, this would entail left and right boundedness of $FA$. 
We suppose moreover that the induced filtration $FK$ on $K$ has non-trivial negative part. Since $K$ has no zero divisors, this is equivalent to saying that $F_{-n}K \neq 0$ for all $n \geq 0$.\\

Separatedness of $FA$ entails the existence of a smallest integer $e \geq 1$ such that $F_{-e} \supsetneq F_{-e-1}$, i.e. $G(A)_{-e} = F_{-e}/F_{-e-1} \neq 0$. 

\begin{proposition}\proplabel{estep}
Let $FA$ be a filtration such that $F^{-}B$ is a left irreducible $FA$-glider, then the filtration is a strong $e$-step filtration, with $e$ defined as above.
\end{proposition}
\begin{proof}
To begin with, consider the subglider
$$\begin{array}{ccccccccccc}
F_0 & \supset & F_{-1} & \supset & \ldots & \supset &  F_{-e} & \supsetneq & F_{-e-1} & \supset & \ldots \\
\cup && \veq && &&\cup  &&\cup&&\\
F_{e}F_{-e} & \supset & F_{e-1}F_{-e} &  \supset & \ldots & \supset & F_0F_{-e} & \supset & F_{-1}F_{-e} & \supset & \ldots
\end{array}$$

For obvious reasons it cannot be trivial of type $T_1$ or $T_2$. Because $F_0F_{-e} = F_{-e} \supsetneq F_{-e-1}$, triviality of type $T_3$ implies that $F_eF_{-e} = F_0$. A similar argument shows that $F_{-e}F_e = F_0$ as well. Irreducibility of $F^-B$ implies furthermore that $F_{-e}F_{-e} = F_{-2e -d}$ for some $d \geq 0$. But we also have 
$$F_{-e} = F_0F_{-e} = F_eF_{-e}F_{-e} = F_eF_{-2e-d} \subset F_{-e-d},$$
and by definition of $e$ it follows that $d = 0$, i.e. $F_{-e}F_{-e} = F_{-2e}$. Suppose now that $F_{-2e} = F_{-3e}$, then we would have 
$$F_{-e} = F_0F_{-e} = F_eF_{-e}F_{-e} = F_eF_{-2e} = F_eF_{-3e} \subset F_{-2e},$$
contradiction. Hence $F_{-2e} \supsetneq F_{-3e}$. This allows us to show that $F_{-e}^3 = F_{-3e}$. Indeed, irreducibility of $F^-B$ implies that $F_{-e}^3 = F_{-3e -d}$ for some $d \geq 0$. We also have
$$F_{-2e} = F_0F_{-2e} = F_eF_{-e}F_{-2e} = F_eF_{-3e-d} \subset F_{-2e-d}.$$
If $d > 0$, then we would have that $F_{-2e} \subset F_{-2e-d} \subset F_{-3e} \subset F_{-2e}$, contradiction. By induction one then shows that for all $n \geq 0$, $F_{-ne} \supsetneq F_{-(n+1)e}$ and $F_{-e}^n = F_{-ne}$.\\
Consider now the subglider
$$F_eF_{-2e} \supset F_{e-1}F_{-2e} \supset \ldots \supset F_0F_{-2e} \supset F_{-1}F_{-2e} \supset \ldots$$
of 
$$F_{-e} \supsetneq F_{-e -1} \supset \ldots \supset F_{-2e} \supset F_{-2e-1} \supset \ldots $$
Since $F_eF_{-2e} = F_eF_{-e}F_{-e} = F_{-e}$ and $F_{-1}F_{-2e} = F_{-e}F_{-2e} = F_{-3e} \subsetneq F_{-2e}$ and triviality of type $T_3$ implies that the associated monotone increasing map is the identity on $\{0,1\ldots, e\}$. In particular, it follows that
$$F_{-e-1} = F_{e-1}F_{-2e} = F_{e-1}F_{-e}F_{-e} \subset F_{-1}F_{-e} = F_{-e}F_{-e} = F_{-2e},$$
whence 
$$F_{-e-1} = F_{-e-2} = \cdots = F_{-2e}.$$
If $F_{-2e} = F_{-2e-1}$, then $F_{-e} = F_eF_{-2e} = F_eF_{-2e-1} \subset F_{-e-1}$, contradiction. Using induction one then shows that

$$F_{-ne} = F_{-ne +e-1}  \quad {\rm for~ all~} n \geq 0.$$

Using similar arguments, one shows by induction that $F_{ne}F_{-ne} = F_0$ for all $n > 0$. We also have
$$F_{e-1} = F_{e-1}F_0 = F_{e-1}F_{-e}F_e \subset F_{-1}F_{e} = F_{-e}F_e = F_0.$$
Both results then allow to show that 
$$F_{ne} = F_{ne +e-1}  \quad {\rm for~ all~} n \geq 0.$$
We can conclude if we show that $F_{ne} = F_e^n$ for all $n \geq 2$. Since we have
$$F_{ne} = F_{ne}F_0 = F_{ne}F_{-e}F_e \subset F_{(n-1)e}F_e,$$
this follows easily using induction.
\end{proof}

In case $A = K$ is a field, we know by \cite[Theorem 2.6]{LiVo} that $F_0K = R$ is a discrete valuation ring and the associated valuation filtration is $FK$ if and only if the associated graded $G(K) = \oplus_{n \in \mathbb{Z}} F_nK/F_{n-1}K$ is a domain. This result allows to show

\begin{proposition}\proplabel{GKdomain}
Let $FK$ be a separated, exhaustive filtration with non-trivial negative part and such that $F^-R$ is an irreducible $FK$-glider. If $e=1$, then $R = F_0K$ is a discrete valuation ring and the filtration corresponds to the valuation filtration.
\end{proposition}
\begin{proof}
Suppose that $G(K)$ is not a domain, whence not semiprime. Since $FK$ is strong, the associated graded ring $G(K)$ is strongly graded. From \cite{Dade} we know that $G(K)$-gr $\cong F_0K/F_{-1}K$-mod is an equivalence of categories, whence there exists $a \in F_0K - F_{-1}K$ and $n > 1$ minimal such that $a^n \in F_{-1}K$. The induced filtration on $F_0Ka$ defines a subglider, which must be trivial of type $T_3$. Because $a \in F_0Aa - F_{-1}A$, this entails that $F_0Aa = F_0A$. In particular,  $a$ is then invertible in $F_0A$, say $ba = 1$ for $b \in F_0A$. But then $a^{n-1} = 1a^{n-1} = ba^n \in F_{-1}A$, contradicting the minimality of $n$. Hence $G(K)$ is a domain. It follows that $F_0K = B$ is a DVR with unique maximal ideal $M$. We have that $F_{-1}K = M^n$ for some $n \geq 1$, but because $F_0K/F_{-1}K$ is a domain, $n = 1$, i.e. $F_{-1}K = M$. If $F^vK$ denotes the valuation filtration, then 
$$F_1K = F_1KF_0K = F_1KMF^v_1K = F_1KF_{-1}KF^v_1K = F^v_1K.$$
and since both filtrations are determined by their degree -1 and degree 1 part, both filtrations agree.
\end{proof}

In \cite[Theorem 3.10]{LiVo}, the authors give a generalization by showing that quasi-simple rings $A$ with exhaustive and separated filtration $FA$ and $G(A)$ a semiprime, Noetherian P.I., ring actually have for $G(A)$ a prime ring. 

\begin{proposition}\proplabel{strongCSA}
Let $A$ be a simple Artinian ring with separated exhaustive filtration $FA$ that induces $FK$ with non-trivial negative part. If $F^-B$ is an irreducible $FA$-glider then $F_0K$ is a discrete valuation ring in $K$. If $FK$ is a strong $f$-step filtration then $F_{nf}A = F_{nf}KF_0A$ for all $n \in \mathbb{Z}$. 
\end{proposition}
\begin{proof}
As in the beginning of the proof of \propref{GKdomain} we can show that $G(A)$ must be a prime ring by using \cite[Theorem 3.10]{LiVo}. Indeed, suppose $G(A)$ is not semiprime. Then there exists an ideal $J \triangleleft F_0A$ such that $J^n \subset F_{-1}A$ for some $n > 1$ and assume $n$ is minimal with this property. The triviality of the subglider $J^{n-1} \supset F_{-1}J^{n-1} \supset F_{-2}J^{n-1} \supset \ldots$ entails that $F_0 = J^{n-1}$ but then it follows that $J \subset F_{-1}$, contradiction.  The center of a prime ring is an integral domain, whence so is $G(K) \subset Z(G(A))$. Hence $FK$ is the filtration associated to a valuation ring $O_v$ on $K$ when considered over $f\mathbb{Z}$ for some $f > 0$. In particular, $f$ is minimal with the property that $F_{-f-1}K \subsetneq F_{-f}K$ and $G(K)$ is $f\mathbb{Z}$-strongly graded. Define the glider $M_j = F_{-f}KF_{f-j}A \subset F_{-j}A$. For $i \leq j$ we indeed have $$F_iAM_j \subset = F_{-f}KF_iAF_{f-j}A \subset F_{-f}KF_{f+i-j}A = M_{j-i}.$$
Irreducibility of $F^-B$ implies that the subglider $M \supset M_1 \supset \ldots$ is trivial of type $T_3$. If the strict monotone increasing map $\alpha$ is not the identity, then $F_{-f}KF_0A = F_{-f-d}A$ for some $d >0$. Since $1 \in F_0A$, this would entail 
$$F_{-f}K \subset F_{-f-d}A \cap K = F_{-f-d}K.$$
By definition of $f$ this yields $d =0$, i.e. $F_{-f}F_0A = F_{-f}A$. By \propref{estep} we know that $FA$ is a strong $e$-step filtration. Since $GK$ is an $f$-step filtration, $e$ divides $f$. Hence
$$F_{-nf}KF_0A = F_{-nf}A \quad \forall n \geq 0.$$
From $F_{-f}KF_0A = F_{-f}A$ it follows that $F_{-f}KF_fA = F_0A$. Consequently
$$F_fKF_0A = F_fKF_{-f}KF_fA = F_0KF_fA = F_fA,$$
and then 
$$F_{nf}KF_0A = F_{nf}A \quad \forall n \geq 0$$
follows. 
\end{proof}

In \cite[Theorem 2.3]{NVo1} the authors introduced completely irreducible gliders and showed that a simple Artinian ring $A$ has a filtration $FA$ with subring $F_0A = A'$ making the negative part into a completely irreducible $FA$-glider if and only if $A$ is a skewfield, $A'$ is a discrete valuation ring and $FA$ is the corresponding valuation filtration. Our approach leads to a generalization, as we do not require the notion of completely irreducibleness.

\begin{proposition}
Let $A$ be a central simple $K$-algebra with exhaustive, separated filtration $FA$ inducing a filtration $FK$ with non-trivial negative part and such that $F^-B$ is an irreducible left glider. Then $A$ is a skew field. 
\end{proposition}
\begin{proof}
Suppose that $A$ is not a skewfield. So there exists a proper left ideal $Av$. Since $FA$ is exhaustive and separated, $v \in F_n - F_{n-1}$ for some $n \in \mathbb{Z}$. The subglider $F_{-n}v \supset F_{-n-1}v \supset \ldots$ of $F^{-}B$ must be trivial. If it is trivial if type $T_1$, then $F_mv = 0$ for some $m \in \mathbb{Z}$. Since $0 \neq F_m \subset \Ann_A(Av)$, it follows that $F_0 = F_{-m}F_m \subset \Ann_A(Av)$, contradiction. If it is trivial of type $T_2$, then $F_mv = F_{m+1}v$ for some $m \in \mathbb{Z}$. By \propref{strongCSA} it follows that $Av = F_0v = F_{-n}v$ for all $n \in \mathbb{Z}$. In particular, $v \in B(FA) = 0$, contradiction. So the subglider must be trivial of type $T_3$, but in this case, $F_{-n}v = F_{-m}$ for some $m \in \mathbb{N}$. Hence $F_{m-n}v = F_0$, which shows that $Av = A$, contradiction. 
\end{proof}

\propref{strongCSA} also reveals a link with the theory of orders in simple algebras. Denote $R = F_0K$. We recall that a full $R$-lattice is a finitely generated $R$-torsion free module $M$ in $A$ such that $KM = A$. By definition, an $R$-order $C$ in $A$ is a subring of $A$ which is also a full $R$-lattice. If $C$ is not contained in any proper $R$-order $D$, we call $C$ a maximal order. We refer the reader to \cite{Rein} for a detailed treatment of the theory of maximal orders.
\begin{corollary}\corlabel{order}
In the situation of \propref{strongCSA}, $B = F_0A$ is an $R$-order in $A$.
\end{corollary}
\begin{proof}
 $B$ is prime since $BK = A$ is a csa and $Z(B) = F_0K = R$ is a DVR hence a Noetherian ring. Moreover, $B$ is a P.I. ring as a subring of a P.I. ring $A$, hence by a result of Formanek, see \cite[Theorem 2]{For}: $B$ is a finitely generated $Z(B)$-module
\end{proof}

Suppose that $R = F_0K$ is a Dedekind domain and $B$ is a maximal $R$-order in $A$. A prime ideal of $B$ is by definition a proper two-sided ideal $P$ in $B$ such that $KP = A$ and such that for every pair of two-sided ideals $S,T$ in $B$ and containing $P$, we have
$$ST \subset P \Rightarrow S \subset P {\rm~or~} T \subset P.$$
In fact, the prime ideals of $B$ coincide with the maximal two-sided ideals of $B$ and if $P$ is such a prime ideal, then $P \cap R \in \Spec(R)$. Vice-versa, for $p \in \Spec(R)$, $P = B \cap \rad(B_p)$ is a prime ideal of $B$, and this yields a one-to-one correspondence. Also, the product of prime ideals in $B$ is commutative and every two-sided ideal of $B$ can be written uniquely as a product of prime ideals. Since $R$ is Dedekind, there are only a finite number of prime (hence maximal) ideals $p_1, \ldots, p_n$. If $P_i$  corresponds to $p_i$, then 
$$p_iB = P_i^{e_i}$$
for some $e_i > 0$, and we call $e_i$ the ramification index at $P_i$. 

\begin{lemma}\lemlabel{fieldjac}
Let $FK$ be an exhaustive, separated filtration, then $F_{-1}K \subset J(F_0K)$.
\end{lemma}
\begin{proof}
Let $N \in \Max(F_0K)$ and let $x \in N$. Since $FK$ is exhaustive there exists $n > 0$ such that $x^{-1} \in F_nK - F_{n-1}K$. Hence $F_{-1}^nx^{-1} \subset F_0$ and it follows that $F_{-1}^n \subset (x) \subset N$. Since $N$ is prime, it follows that $F_{-1} \subset N$ and the result follows.
\end{proof}

\begin{lemma}\lemlabel{Ded}
Let $R$ be Dedekind with quotient field $K$ and $B$ be a maximal $R$-order in $A$. If $P \neq Q$ are prime ideals of $B$, then 
$$PQ \cap K = (P \cap K) \cap (Q \cap K) = (P \cap R) \cap (Q \cap R).$$
\end{lemma}
\begin{proof}
Follows since $PQ \cap K \subset (P \cap R) \cap (Q \cap R)$, by the correspondence of prime ideals and the fact that every two-sided ideal can be written uniquely as the product of prime ideals.
\end{proof}

We denote the ceil function by $\ceil*{}$.
\begin{lemma}\lemlabel{ceil}
Let $e > 0$ and $k,l \in \mathbb{Z}$. We have the inequality
$$\ceil*{\frac{k}{e}} + \ceil*{\frac{l}{e}} \geq \ceil*{\frac{k+l}{e}},$$
and we have strict inequality if and only if
$$k = k_ee + i, ~0 < i \leq e-1,$$
$$ l = l_e  e + j,~ 0 < j \leq e-1$$
and $2 \leq i + j \leq e$.
\end{lemma}
\begin{proof}
By writing $k = k_e e +i, l = l_ee+j$ with $0 \leq i,j \leq e-1$, we have 
$$\ceil*{\frac{k}{e}} = k_e + 1 - \delta_{0,i}, \quad \ceil*{\frac{l}{e}} = l_e + 1 - \delta_{0,j}.$$
The statements now follow easily.
\end{proof}

\begin{theorem}\thelabel{maxorder}
Let $A$ be a CSA over $K$ with filtration $FA$ such that $F_0A = B$ is a maximal $R = F_0K$-order, with $R$ a Dedekind domain. Then $FA$ is strong if and only if 
$$F_{-1}A = P_1^{k_1}\ldots P_{n}^{k_n},$$
with $e_i | k_i$ for all $1 \leq i \leq n$.
\end{theorem}
\begin{proof}
By the theory of maximal orders, $F_{-1}A = P_1^{k_1}\ldots P_{n}^{k_n}$ for some $k_i \geq 0$. \lemref{fieldjac} and \lemref{Ded} entail that
$$F_{-1}K = p_1^{\ceil*{\frac{k_1}{e_1}}}\ldots p_n^{\ceil*{\frac{k_n}{e_n}}}.$$
If $e_i | k_i$ for all $1 \leq i \leq n$, then $m \ceil*{\frac{k_i}{e_i}} = \ceil*{\frac{mk_i}{e_i}}$ for all $1 \leq i \leq n$ and $m \in \mathbb{Z}$. It follows that $FK$ is strong. Conversely, suppose that $FK$ is strong, but that $\ceil*{\frac{k_1}{e_1}} > \frac{k_1}{e_1}$, i.e. $k_1 = le_1 + j$ for some $0 < j < e_1$. Since $FK$ is strong, we have for every $m$
$$m \ceil*{\frac{k_1}{e_1}} = \ceil*{\frac{mk_1}{e_1}},$$
i.e. $(m-1)e_1 < mj < me_1 - m$. It follows that for every $m \in \mathbb{Z}$
$$\frac{m-1}{m}e_1 < j \leq e_1 - 1,$$
which is a contradiction. 
\end{proof}

\section{The Glider-Brauer-Severi variety of a field}\seclabel{GBSfield}

Let $A$ be a central simple algebra of degree $n^2$ over a field $K$. For every $1 \leq r \leq n$ one may define the (generalized) Brauer-Severi variety $\BS_r(A)$ as the variety of left ideals of reduced dimension $r$ in $A$, see \cite{involutions}. Such an ideal is represented by a non-zero $rn$-vector $u_1 \wedge \ldots \wedge u_{rn} \in \bigwedge^{rn} A$. If $(e_i)_{1 \leq i \leq n^2}$ denotes a basis of $A$, then the $rn$-dimensional subspace represented by $u_1 \wedge \ldots \wedge u_{rn} \in \bigwedge^{rn} A$ is a left ideal of reduced dimension $r$ if and only if it is preserved under left multiplication by $e_1,\ldots, e_{n^2}$. When $r=1$, we obtain the classical Brauer-Severi variety $\BS(A)$ and we see that 
$$\BS(A) = \{ L \leq A {\rm ~irreducible~left~ideal}\}.$$
If we consider the set of all left ideals of $A$ then we obtain 
$$\sqcup_{r=1}^n \BS_r(A).$$

\begin{definition}
Let $FR$ be a filtered ring with subring $F_0R = S$. We call $L \subset R$ a left glider ideal if $L$ is a left $FR$-glider, in particular $L$ is a filtered left $F^-S$-module.
\end{definition}

\begin{definition}
Let $A$ be a $K$-algebra. For a filtration $FA$ on $A$, we define the glider-Brauer-Severi variety associated to the filtration $FA$ 
$$\GBS_F(A) = \{{\rm irreducible~left~} FA{\rm-glider~ ideals ~of ~} A\}$$
And more generally, we define the glider-Brauer-Severi variety 
$$\GBS(A) = \bigcup_{FA {\rm~filtration}} \GBS_F(A).$$
\end{definition}

As mentioned in the introduction, the theory of irreducible gliders highly depends on the type of filtrations one is working with. For example, when working with right bounded filtrations, we recall

\begin{lemma}\cite[Lemma 2.5]{CVo1}\lemlabel{el}
Let $M$ be a (weakly) irreducible $FA$-glider such that $M \neq B(M)$, then there is an $e \in \mathbb{N}$ such that $M_e \neq B(M)$ and $e$ is maximal as such. For this $e$, we have that $M_i = F_{e-i}AM_e$, for $0 \leq i \leq e$.
\end{lemma}

In particular, irreducible gliders are of finite essential length $d$ and if they have zero body, their structure is determined by the simple $F_0A$-module $M_d$. We also recall from \cite[Lemma 2.1]{CVo2} that whenever $M\supset M_1 \supset \ldots$ is irreducible, then so is $M_m \supset M_{m+1} \supset \ldots$ for any $m \geq 0$.\\

For the non-bounded separated, exhaustive filtrations on central simple algebras, the question how irreducible gliders look like, has not been answered up to now. The following lemma will be useful for tackling this problem

\begin{lemma}\lemlabel{irrinfinite}
If $M$ is irreducible of infinite essential length, then for all $i \geq 0$ either $M_i = M_{i+1}$ or $M_i/M_{i+1}$ is an irreducible left $F_0A$-module.
\end{lemma}
\begin{proof}
Suppose that $M_i > N > M_{i+1}$ is a proper left $F_0A$-module. Then 
$$\begin{array}{ccccccccccc}
 M_i & \supsetneq & M_{i+1} & \supset & M_{i+2} & \supset & \ldots \\
\cup  &&\cup&& \cup &&\\
 N & \supset & F_{-1}N& \supset & F_{-2}N & \supset & \ldots
\end{array}$$
would be a non-trivial subglider of the irreducible glider $M_i \supset M_{i+1} \supset \ldots$.
\end{proof}

Throughout the rest of this section, we focus on $A = K$ and determine its Glider-Brauer-Severi variety $\GBS(K)$. For non-bounded filtrations, it appears that irreducible gliders have infinite essential length. Indeed, $K$ being a field implies that an irreducible left glider ideal $M$ has zero body. So if $\el(M)= d$, then $F_{-d-1}KM = 0$, contradiction. Suppose that $M \in \GBS_F(K)$. \lemref{irrinfinite} entails that there exists $n \in \mathbb{Z}, i \in \mathbb{N}$ such that $M_i = F_nK$. Indeed, consider on $KM = K$ the filtration given by $F_{-n}KM = M_n, F_n(KM) = \sum_{j \geq 0}F_{j+n}M_j$ for $n \geq 0$. Suppose first that $1 \in M_i \setminus M_{i+1}$. Then $F_iK \subset M$. The submodule
$$\frac{F_iK}{F_iK \cap M_{1}} \leq \frac{M}{M_{1}}$$
cannot be zero, whence $F_iK  = M$. If $1 \in F_n(KM) - F_{n-1}(KM)$ for some $n > 0$, then we can write
$$1 = \sum f_{j+n}m_j$$
for some $f_{j+n} \in F_{j+n}K, m_j \in M_j$. In this case, $F_{-n}K \subset M$, and since $B(M) = 0$, there exists a maximal $i \geq 0$ such that $F_{-n}K \subset M_i$. By irreducibility as $F_0K$-module of $M_i/M_{i+1}$ it follows that $F_{-n}K = M_i$.\\

To continue, we again use $B(M) = 0$ for the existence of an $i_1 \geq i$ maximal such that $M_{i_1} \supset F_{n-1}K$. And by \lemref{irrinfinite} we actually have equality. This leads to a sequence $i \leq i_1 \leq i_2 \leq \ldots$ such that
$$M_{i_k} = F_{n - k}$$
and for all $0 \leq l \leq i_{k+1} - i_k$ we obtain the equalities
$$M_{i_k + l} = F_{i_{k+1} - l}F_{n-k}.$$
In particular, we have that 
$$l_{F_0K}(M_{i_k}/M_{i_{k+1}}) \leq i_{k+1} - i_k.$$

In particular, if $FK$ is a strong filtration, then $i_{k+1} = i_k + 1$ for all $k$. It follows that there exists an $n \in \mathbb{Z}$ such that $M$ equals
$$(F_nK)_*: \quad F_nK \supset F_{n-1}K \supset F_{n-2}K \supset \ldots$$

\begin{example}\exalabel{DVR}
Let $O_v \subset K$ be a discrete valuation ring and $F = F^v$ the $v$-adic filtration on the field of fractions $K$. Let $L \subset \Omega \in \GBS_{F^v}(K)$. Suppose that $L_n = L_{n+1}$ for some $n \geq 0$, then $F_1L_n = F_1L_{n+1} \subset L_n \subset F_1L_n$, whence $F_1L_n = L_n$ and it follows that $F_nL_n = L_n = K$. But then it follows that $K = F_{-1}L_{n+1} \subset L_{n+2}$ and $L$ is just the glider
$$K \supset K \supset K \supset \ldots,$$
which has the non-trivial subglider $O_v \supset F_{-1} \supset \ldots$. In particular, it follows that $L_0 \supsetneq L_1$. Let $l \in L_0 - L_1$. Then $O_vl \supset F_{-1}l \supset F_{-2}l \supset \ldots$ is a trivial subglider of type $T_3$, and because $l \notin L_1$, it follows that $L = O_vl$. In particular, $L$ is a fractional ideal and since $O_v$ is a local Dedekind domain, $L = F_n$ for some $n \in \mathbb{Z}$. In fact, one shows that all $L_i$ are fractional ideals and so $L_1 = F_{n+m}$ for some $m < 0$. If $m < -1$, then  
$$\Omega \supset F_{n-1} \supset L_1 \supset L_2 \supset L_3 \supset \ldots$$
would be a non-trivial subglider. Hence $m = 1$ and actually one can show that $L_i = F_{-n-i}$ for all $i \geq 0$. Hence we have shown that
$$\GBS_{F^v}(K) = \{(F^v_n)_* ~\vline~ n \in \mathbb{Z} \},$$
where $F^v_n$ has the chain $F^v_n \supset F^v_{n-1} \supset F^v_{n-2} \supset \ldots$
\end{example}

\begin{example}
Let $p,q$ be two distinct prime numbers. Consider the filtration on $\mathbb{Q}$ defined by $F_0 = \mathbb{Z}_S$ where $S$ is the multiplicatively closed set generated by all prime numbers except for $p$ and $q$. The negative part is the $(pq)$-adic filtration and for $i > 0$ we set
$$F_i = \sum_{k,l \leq i} \mathbb{Z}_S\frac{1}{p^kq^l}.$$
This is a strong filtration. We have that $\Omega = \mathbb{Q}$ whence the $M$-chain must be
$$F_n \supset F_{n-1} \supset F_{n-2} \supset \ldots$$
starting at some $n \in \mathbb{Z}$, i.e. the glider is $(F_n)_\g$. This is however not irreducible since the consecutive quotients are not simple $\mathbb{Z}_S$-modules. Indeed, for any $m \in \mathbb{N}$
$$ (pq)^n \supsetneq (p^{n+1}q^n) \supsetneq (pq)^{n+1}.$$
\end{example}

These examples indicate that for strong filtrations on fields, non-emptiness of the GBS is equivalent to the $F_0K$ being a discrete valuation ring. Indeed, we have

\begin{proposition}\proplabel{fieldstrong}
Let $K$ be a field with exhaustive, separated strong filtration $FK$. Then $\GBS_F(K) \neq \emptyset$ if and only if $F_0K = O_v$ is a DVR and $F = F^v$ is the associated $v$-adic filtration.
\end{proposition}
\begin{proof}
Suppose that $\GBS_F(K) \neq \emptyset$ and let $M \in \GBS_F(K)$. Suppose first that $F_0 = O_v$ is a DVR. There exists $a > 0$ such that $F_{-1} = (\pi^a)$ if $\pi$ is a uniformising parameter of $F_0$. By the structure of the elements in the $\GBS_F(K)$ for strong filtrations, there exists an $n \in \mathbb{Z}$ such that $M = F_n$. If $n \geq 0$, then 
$$F_0 \supset F_{-1} \supset F_{-2} \supset \ldots$$
is also irreducible and \propref{GKdomain} entails that the filtration corresponds to the valuation filtration. If $-n > 0$ then $M = F_n = (\pi^{an})$. Since $FK$ is strong, $M_1 = F_{n-1} = (\pi^{a(n+1)})$. Because $M /M_1 \cong (\pi^{a})$ is simple, $a= 1$ and it follows that $FK = F^vK$. Conversely, suppose that $F_0$ is not a DVR. In particular, when $M = F_nK$ as before, the $n$ must be strictly smaller than $0$ by \propref{GKdomain}. Let $y \in F_0 \setminus F_{-1}$. Then $y^{-1} \in \dot{F}_m$ for some $m \geq 0$. We want to show that $m = 0$. The glider
$$y^{-1}F_{-1}^{m+n} \supset y^{-1}F_{-1}^{n+m+1} \supset \ldots$$
is a subglider of
$$F_{-1}^n \supset F_{-1}^{n+1} \supset \ldots,$$
whence must be trivial of type $T_3$. Hence there exists $r \geq 0$ such that $y^{-1}F_{-1}^{m+n} = F_{-1}^{n+r}$. It follows that $y^{-1} \in F_{-1}^{n+r}F_{-1}^{m+n} = F_{m+n -n -r} = F_{m-r}$, whence $m - r \geq m$ or $r \leq 0$. So $r = 0$ and $y^{-1}F_{-1}^{m+n} = F_{-1}^n$ from which $y^{-1}F_{-1}^m = F_0$ follows. In particular we obtain $1 = y^{-1}x$ for some $x \in F_{-1}^m$. However, $x$ must be $y$ and it follows that $y \in F_{-1}^m$. Finally, we started with $y \in \dot{F}_0$, $m$ must indeed be equal to 0. This shows that $F_0$ is local with maximal ideal $F_{-1}$. However, since $F_0$ is not a DVR, $F_{-1}$ is not principal. Since $FK$ is strong, there exists $x \in F_{-1} - F_{-2}$. The subglider
$$\begin{array}{ccccc}
F_{-n} & \supset & F_{-n-1} & \supset & \ldots \\
\cup && \cup  &&\\
xF_{-n+1} & \supset  & xF_{-n} & \supset &  \ldots
\end{array}$$
must be trivial of type $T_3$, so $xF_{-n+1} = F_{-n -r}$ for some $r \geq 0$, or $(x) = F_{-r-1}$. It follows that $r = 0$, which contradicts the principality of $F_{-1}$. Hence $\GBS_F(K) = \emptyset$.\\
\end{proof}

Together with \exaref{DVR} we obtain that when running over all non-bounded, separated, exhaustive strong filtrations $FK$ on $K$, we have
\begin{equation}\label{GBSfield}
\bigcup_{FK}\GBS_F(K) = R(K) \times \mathbb{Z},
\end{equation}
where $R(K)$ denotes the Riemann surface of $K$. Filtrations that are not strong can also have non-zero glider Brauer-Severi variety. Indeed, one can for example consider a DVR $R = O_v$ with maximal ideal $M = (x)$. The positive part on $K$ is just defined by the standard filtration generated by $F_1K = R + Rx^{-1}$. And for the negative part, one takes 
$$R \supset M^2 \supset M^3 \supset \ldots$$
In fact $\GBS_F(K) = \GBS_{F^v}(K)$. 
\begin{proposition}
Let $K$ be a field with exhaustive, separated filtration $FK$. If $\GBS_F(K) \neq \emptyset$, then there exists $n \geq 0$ such that for all $m \geq n$, $F_{-m}$ is a principal ideal of $F_0$.
\end{proposition}
\begin{proof}
By the structure of elements of the glider Brauer-Severi variety, we know that $M = F_lF_m$ for some $l \geq 0$ and $m\in \mathbb{Z}$. In particular, we know that an $F_{-n}$ for $n \geq 0$ appears as an $M_i$. Take $a \in M_i$, then 
$$\begin{array}{ccccccccccc}
F_{-n} & \supset & F_{i_1-1}F_{-n-1} & \supset & \ldots & \supset & F_1F_{-n-1} & \supset & F_{-n-1} & \supset & \ldots \\
\cup && \cup  && && \cup && \cup \\
(a) & \supset  & (a)F_{-1} & \supset &  \ldots & \supset & (a)F_{-i_1+1} & \supset & (a)F_{-i_1} & \supset & \ldots
\end{array}$$
is a subglider of type $T_3$. In particular, there exists $k_a \geq 0$ and $0 \leq l_a \leq i_{k_a}$ such that $(a) = F_{l_a}F_{-n-k_a}$.  If $(a) \subset F_{-n}$ is proper, then there exists $a_1 \in F_{-n} - (a)$, and there exist $k_{a_1} \geq 0$ and $0 \leq l_{a_1} \leq i_{k_{a_1}}$ such that $(a_1) = F_{l_{a_1}}F_{-n-k_{a_1}}$. It must hold that $k_{a_1} \leq k_{a}$ and if equality holds, then $l_{a_1} > l_{a}$. In particular $(a) \subsetneq (a_1) \subset F_{-n}$. If the last inclusion is proper, then we can continue this argument, which must stop by the restrictions on the $k_{a_j}, l_{a_j}$. This shows that $F_{-n}$ is indeed principal. The glider starting from $F_{-n-m}$ for any $m \geq 0$ remains irreducible, and the reasoning above shows that $F_{-n-m}$ is again principal.
\end{proof}

\begin{lemma}\lemlabel{irredfilt}
Let $FK$ be filtration on a field $K$. If $M \in \GBS_F(K)$ with $M = F_0K = R$, then the negative chain $F_{-n}R = M_n$ defines a ring filtration on $R$.
\end{lemma}
\begin{proof}
Let $i > 0$ and consider the subglider
$$\begin{array}{ccccccccccc}
M & \supset & M_1 & \supset & \ldots & \supset & M_i & \supset & M_{i+1} & \supset & \ldots\\
\cup & & \cup &&&& \cup && \cup &&\\
MM_i & \supset & M_1M_i & \supset & \ldots & \supset & M_iM_i & \supset & M_{i+1}M_i & \supset & \ldots
\end{array}$$
Since $K$ is a field it cannot be of type $T_2$. Suppose that $M_nM_i = M_{n+m}M_i$ for some $n$ and all $m > 0$. It would follow that $F_mKM_nM_i = F_mKM_{n+m}M_i \subset M_nM_i$. Exhaustivity of $FK$ implies that $M_nM_i$ is a $K$-vector space in $K$, i.e. $M_nM_i = K$, contradiction. The subglider is thus of type $T_3$. Because $MM_i = M_i$, $\alpha(0) \geq i$, which implies that for every $j  > 0$, $\alpha(j) > i +j$, which amounts to saying that
$$M_jM_i = M_{\alpha(j)} \subset M_{i+j},$$
or $$F_{-j}RF_{-i}R \subset F_{-i-j}R,$$
proving the claim.
\end{proof}

\begin{proposition}\proplabel{associatedfil}
Let $FK$ be a non-strong filtration such that $\GBS_F(K) \neq \emptyset$ and $F_0K = R$. Then there exists a strong $e$-step filtration $F^sK$ on $K$ such that $F^+K = F^{s,+}K$ and $F^{-}R$ is a trivial subglider of type $T_3$ of $F^{s,-}K$.
\end{proposition}
\begin{proof}
Let $M \in \GBS_F(K)$. Up to considering the irreducible glider $M_n \supset M_{n+1} \supset \ldots$, there exists an $m \in \mathbb{Z}$ such that $M = F_mK$. Suppose that $m < 0$. Let $x \in F_{-m}K$, then since $FK$ is exhaustive, $x^{-1} \in F_nK$ for some $n > 0$. Since we are working in a field $K$, the glider
$$xM \supset xM_1 \supset xM_2 \supset \ldots$$
is also irreducible. And because $1 \in xM$, we have $R \subset xM$ and it follows that there appears an $F_rK$ as an $xM_j$ for some $r \geq 0$, hence also $F_0K = R$ appears. In particular, whenever $\GBS_F(K) \neq \emptyset$, we can find an irreducible glider starting with $M = F_0K = R$. By the structure of irreducible gliders, we know that there exists $n > 0$ such that $M_{n} = F_{-1}K$, which leads to $F_0K = F_nKF_{-1}K$. So actually $n > 1$. The previous lemma entails that the $M$-chain defines a negative ring filtration on $R$. Define the filtration $F^sK$ on $K$ by 
$$F^s_{n}K = F_nK, \quad F^s_{-n}K = M_n, \quad n \geq 0.$$
\propref{estep} entails that $F^sK$ is a strong $e$-step filtration. Since $n$ is the smallest number such that $F_nKF_{-1}K = R$, it follows that $e \leq n < 2e$. And since, $F_{n-1}KF_{-1}K < R$, $n$ actually equals $e$. 
\end{proof}

The $e$ in the above proposition is determined by the positive filtration $F^+K$, i.e. it is the smallest number $m \geq 1$ such that $F_{m}K \supsetneq F_{m-1}K$. 

\begin{theorem}\thelabel{GBSfield}
Let $K$ be a field, then
$$\GBS(K) = R(K) \times \mathbb{Z}.$$
\end{theorem}
\begin{proof}
Let $FK$ be a filtration. If $\GBS_F(K) \neq \emptyset$, then $\GBS_F(K) = \GBS_{F^s}(K)$, so the result follows by \eqref{GBSfield}. 

\end{proof}

\section{The relative Glider-Brauer-Severi variety for a central simple algebra}\seclabel{GBScsa}

In this section we determine the relative Brauer-Severi variety $\GBS(A)$ of a central simple algebra $A$ over a field $K$, meaning that we run over all filtrations $FA$ extending some fixed filtration $FK$, i.e. satisfying $F_nA \cap K = F_nK$ for all $n \in \mathbb{Z}$. Before we put some additional conditions on $FK$, we prove the following.

\begin{lemma}\lemlabel{infinitelength}
Let $FA$ be a filtration extending $FK$. If $M \in \GBS_F(A)$ and $M$ is not a left $A$-module, then $\el(M) = +\infty$. 
\end{lemma}
\begin{proof}
Suppose that $\el(M) = d < +\infty$. It follows that $F_{-d-1}M \subset B(M)$. However, since $ 0 \neq F_{-d-1}K \subset F_{-d_1}$ contains invertible elements, it follows that $M \subset B(M) \subset M$, contradiction. 
\end{proof}

\begin{corollary}\corlabel{infinitelength}
In the situation of the previous lemma, if $N = M_i$ is a left $A$-module for some $i \in \mathbb{N}$, then $M_i = N = M$.
\end{corollary}
\begin{proof}
The previous lemma entails $\el(M) = +\infty$, whence we may assume that $M_i \supsetneq M_{i+1}$. The subglider
$$\begin{array}{ccccccccccc}
M & \supset & \ldots & \supset &  M_i & \supsetneq & M_{i+1} & \supset & \ldots \\
\cup && &&\cup  &&\cup&&\\
N & \supset  & \ldots & \supset & N & \supsetneq & M_{i+1} & \supset & \ldots
\end{array}$$
must be trivial of type $T_3$ and $\alpha$ the identity map.
\end{proof}

\begin{proposition}\proplabel{semisimple}
Suppose we are in the situation of \lemref{infinitelength}. If $N \subsetneq M$ is a left $A$-module, then $N \subset B(M)$.
\end{proposition}
\begin{proof}
\corref{infinitelength} entails $N \neq M_i$ for all $i \geq 0$. Irreducibility of $M$ entails that $N \supset N \cap M_1 \supset N \cap M_2 \supset \ldots$ must be trivial and it cannot be of type $T_3$. Since $A$ is semisimple, triviality of type $T_2$ would entail that $AM = N \oplus V$ for some left $A$-module $V$ and $M = N \oplus W$ for $W \subset V$ a left $F_0A$-submodule. Since $M_n \subset W$ it follows that $W \cap M \supset W \cap M_1 \supset \ldots$ would be a non-trivial subglider. Hence the subfragment is trivial of type $T_1$, i.e. there exists $n > 0$ such that $N \cap M_n = N \cap M_{n+m}$ for all $m \geq 0$. Let $x \in N$ and $k \geq n \in \mathbb{N}$. Then $F_{-k}N \subset N \cap M_k$. Since $F_{-k}K \neq 0$, take $s_k \in F_{-k}K$. If $t_ks_k = 1$ with $t_k \in F_{\deg(t_k)}K$, then 
$$x = t_ks_kx \in t_k(N \cap M_k) = t_k(N \cap M_{k+\deg(t_k)}) \subset M_k,$$
which shows that $N \subset M_k$ for all $k \geq n$, and $N \subset B(M)$ follows.
\end{proof}

\begin{proposition}\proplabel{principal}
Let $A$ be a csa over $K$ with exhaustive, separated filtration $FA$ inducing $FK$ with non-trivial negative part. If $M \in \GBS_F(A)$ then $B(M) = 0$ and $AM$ is an irreducible principal left $A$-module.
\end{proposition}
\begin{proof}
We first show that $AM/B(M)$ is a principal left $A$-module. Suppose that $AM/B(M) = A\ov{v} + A\ov{w}$ is generated by two elements. The proof for more generators is an easy modification. Let $v, w$ be lifts for $\ov{v}, \ov{w}$ respectively.  If $F_nv \subset M$ for all $n \in \mathbb{Z}$, then $Av \subsetneq M$ and $Av \subset B(M)$ by \propref{semisimple}, which would lead to $\ov{v} = 0$. Reasoning similarly for $Aw$ we have shown that there exist $n,m$ maximal such that $F_nv \subset M$ and $F_mw \subset M$. Since $\ov{v} \neq 0$ there exists $i \geq 0$ minimal such that $F_nv \subset M_i$ and $F_nv \not\subset M_{i+1}$. In particular, $M_{i+1} \subsetneq M_i$ and \lemref{irrinfinite} entails that $M_i = F_nv$. Similarly, we arrive at the equality $F_mw = M_j$ for some $j$. Without loss of generality, we may assume $i \leq j$, i.e. $F_nv \subset F_mw$. However, since $F_nK \subset F_n$ contains invertible elements, it follows that $v \in Aw$, which entails $AM/B(M)$ is principal.  Since $A$ is semisimple, $\Omega = AM = B(M) \oplus Av$ for some left $A$-module $Av$. However, the irreducible left glider $M$ then becomes a strong glider direct sum - i.e. the sum is direct on every level - 
$$B(M) \dot{\oplus} \Big(N \supset N_1 \supset \ldots\Big),$$
where $N = F_nv$ for some $n \in \mathbb{Z}$. Irreducibility then shows that $B(M)$ must be zero and $Av$ must be an irreducible left ideal.  
\end{proof}

The previous proposition shows that $AM = Av$ for some $Av \in \BS(A)$, the classical Brauer-Severi variety. One can perform the similar reasoning we did for fields to show that there exists $n \in \mathbb{Z}, i \in \mathbb{N}$ such that $M_i = F_nAv$. If we further assume that $FA$ is strong, then $M$ becomes the glider
\begin{equation}\label{formstrong}
(F_mAv)_*: \quad F_mAv \supset F_{m-1}Av \supset F_{m-2}Av \supset \ldots
\end{equation}

On the field, we put a valuation filtration corresponding to a DVR $F_0K = O_v = R$. Since $FA$ extends $FK$, we have for $n \geq 0$, $F_{-n}KF_nA \subset F_0A$ and since $FK$ is strong, it follows that $F_nA \subset F_nKF_0A$. The other inclusion is obvious, so $F_nA = F_nKF_0A$ for all $n \in \mathbb{Z}$. This shows that $FA$ is strong as well, and the irreducible left glider ideals are of the form \eqref{formstrong}. We denote the relative Glider-Brauer-Severy variety with regard to $FK$ by $\GBS^{FK}(A)$. 

\begin{proposition}\proplabel{relativecsa}
Let $A$ be a csa with $FA$ extending $FK$ a valuation filtration with $F_0K$ a DVR and put $F_0A = B$. Let $M$ be non-empty in $\GBS_F(A) = \GBS^{FK}(A)$, then $AM = KM = Av$ with $Av \in \BS(A)$, $v$ any generator of an $A$-module. Then there is an $n \in \mathbb{N}$ such that $M_i = F_{n-i}Av$ and $M_i = F_{n-i}Av \cap Kv = F_{n-i}Kv$ is an irreducible $FK$-glider in $Kv$. Thus the $F_{n-i}Kv$ define an element of $\GBS_F(K)$. As a consequence we get $\GBS^{FK}(A) = \BS(A) \times \mathbb{Z}$ (as sets).
\end{proposition}
\begin{proof}
We have $AM = KM = Av$ by \propref{principal}, and also $B(M) = 0$ and $Av \in \BS(A)$. Since $v \in KM, \lambda v \in M$ for some $\lambda \in K^*$ so we may assume $v$ is chosen so that $v \in M$. We have a subglider
$$\begin{array}{ccccccccc}
M = M_0 & \supset & M_1 & \supset & \ldots & \supset & M_n & \supset & \ldots\\
\cup && \cup &&&& \cup &\\
Bv & \supset & F_{-1}Av & \supset & \ldots & \supset & F_{-n}Av & \supset & \ldots
\end{array}$$
Since $FK$ is strong, $FA$ is strong too and it follows for $n$ that $F_nKF_0A = F_nA$. Since $B(M) = 0, B(Bv) = 0$ too and so no $F_{-n}Av \supset F_{-n}Kv$ is zero, thus the glider $BV$ above has to be trivial of type $T_3$, say $Bv = M_n$ for some $n \geq 0$. Then $M = F_nAM_n = F_nKM_n$ yields $M_i = F_{n-i}Av$ for $i \geq 0$. Since $FK$ is a valuation filtration, we know that $F_nK \supset \ldots \supset F_{n-i}K \supset \ldots$, is an irreducible $FK$-glider hence an element of $\GBS_F(K) \subset \GBS(K)$. But then $K \cong Kv$ yields that $F_nKv \supset F_{n-1}Kv \supset \ldots$, is also an irreducible glider in $Kv$. Now consider the chain
$$Kv  \supset M_i \cap Kv \supset F_{n-i}Kv.$$
Then $B(M_i \cap Kv)$ contains $BF_{n-i}Kv = F_{n-i}Av = M_i$, thus $B(M_i \cap Kv) = M_i$. If for some $M_i$ we would have $M_i \cap Kv \supsetneq F_{n-i}Kv$ then since $FK$ is a valuation filtration $M_i \cap Kv = F_mKv$ with $m > n -i$ and then $M_i = B(M_i \cap Kv) = BF_mKv = F_mAv \supset BF_{n-i}Kv = M_i$ with $m > n-i$. Thus $F_{-1}KM_i = M_i$ and since $FK$ is strong $F_{-1}K \subset J(F_0K)$. Now $B$ is prime since $BK$ is a csa and $Z(B) = F_0K$ is a DVR hence a Noetherian ring. The same reasoning as in the proof of \corref{order} shows that  $B$ is a finitely generated $Z(B)$-module, and thus $F_{-1}KM_i = M_i$ with $M_i = BF_{n-i}Kv$ also a finitely generated $F_0K$-module, yields $M_i = 0$, a contradiction. Hence for all $M_i \neq 0$ we thus get $M_i \cap Kv = F_{n-i}Kv$. If we choose another generator $w$ for $Av$, thus $Av = Aw$ then $w$ in some $F_mA - F_{m-1}A$ and the change from $v$ to $w$ in the foregoing is just coming down to a shift by $F_{m-n}K$ in the glider of $K$ we found, consequently the $Av$ define the element of $BS(A)$, the $FA$ corresponds to the valuation $F_0K$ and the choice of generator for $KM = Av$ yields a shift of the $K$-glider $F_0K \supset F_{-1}Kv \supset \ldots \supset F_{-n}K \supset \ldots$. Since we may replace $v$ by any $\lambda v$ with $\lambda \in K^*$, it is clear that all shifts over $\mathbb{Z}$ appear and thus we finally arrive at 
$$\GBS^{FK}(A) = \BS(A) \times  \mathbb{Z}.$$
as sets.
\end{proof}

By the previous section we know that we obtain the whole off $\GBS(K)$ by running over all valuation filtrations on $K$. Hence, we define the relative Glider-Brauer-Severi variety 
$$\GBS^K(A) = \bigcup_{FK {\rm~valuation~filtration}} \GBS^{FK}(A).$$

\begin{corollary}\corlabel{GBScsa}
We have a bijection as sets:
$$\GBS^K(A) = \BS(A) \times R(K) \times \mathbb{Z} = \BS(A) \times \GBS(K).$$
\end{corollary}

\corref{order} indicated a link between irreducible gliders and the theory of orders. This indication also manifests itself in the relative Glider-Brauer-Severi varieties of central simple algebras when $FK$ is a valuation filtration. Before we exploit this, let us recall a few facts about maximal orders over a DVR $R$. For a detailed overview, the reader is referred to \cite{Rein}. A maximal $R$-order in a central simple algebra $M_n(D)$ is conjugate to $M_n(\Lambda)$ where $\Lambda$ is the unique maximal $R$-order in $D$, a central division algebra. In fact, $\Lambda$ is the integral closure of $R$ in $D$. The 1- and 2-sided ideals in $\Lambda$ are of the form $\pi^i \Lambda$. It follows that maximal $R$-orders in $M_n(D)$ are also principal ideal rings and there exists a uniformizer $\Pi$ generating these ideals.

\begin{proposition}\proplabel{e1}
Let $B$ be a maximal $R$-order in $A$ with filtration $FA$ extending $FK$. If $\GBS_F(A) \neq \emptyset$ then the ramification index $e$ is 1.
\end{proposition}
\begin{proof}
If $\Pi, \pi$ are uniformizers for $B, R$ respectively, then $\pi B = \Pi^e B$. If $M \in \GBS_F(A)$, then there exists $m \in \mathbb{Z}$, $Av \in P(A)$ such that $M = F_mKBv = \pi^mBv = \Pi^{me}Bv$ and $M_1 = F_{m-1}KBv = \pi^{m-1}Bv = \Pi^{(m-1)e}Bv$. The quotient $M/M_1$ must be a simple left $B$-module, hence $e = 1$. 
\end{proof}

\begin{example}
Consider $R = \mathbb{Z}_{(2)}[i] \subset K = \mathbb{Q}(i)$. Then $M_n(R)$ is a maximal $R$-order in $A = M_n(K)$, with ramification index 2 at the prime ideal $P = M_n((1+i)R)$: $P^2 = M_n(2R) = (2)M_n(R)$. Let $Av \leq_l A$ correspond to $v \in \mathbb{P}^n$, i.e. to 
$$\Bigg\{ \begin{pmatrix} (\sum_{i=1}^n a_{i1}) v \\ (\sum_{i=1}^n a_{i2}) v \\ \ldots \\ (\sum_{i=1}^n a_{in}) v \end{pmatrix}~\vline~ a_{ij} \in K \Bigg\}.$$
From this it follows readily that $M_n(R)v/M_n(2R)v \cong M_n(R/2R) \ov{v}$ is not irreducible as left $M_n(R)$-module.
\end{example}

There is an obvious way to also reach the maximal orders with ramification index strictly bigger than 1. Indeed, we can vary allow the valuation filtration $FK$ to have a higher step size $f > 1$. If $FA$ is a filtration extending $FK$, then we only have that $F_{nf}A = F_{nf}KF_0A$ for all $n \in \mathbb{Z}$. The filtration $FA$ is therefore not strong in general. If we do impose the strongness condition, we get a generalization of \propref{e1}.

\begin{proposition}
Let $FK$ be the strong $f$-step filtration associated to $R = O^v$. Let $FA$ be a strong filtration extending $FK$ such that $F_0A = B$ is a maximal $R$-order. Then  $\GBS_F(A) \neq \emptyset$ implies $e = f$ is the ramification index at $P$.
\end{proposition}
\begin{proof}
The filtration being strong implies that $M = (F_nA)_\g v \in \GBS_F(A)$, for some $n \in \mathbb{Z}$ and $v \in A$. There exists $i \geq 0$ and $m \in \mathbb{Z}$ such that $M_i = F_{-mf}Av = F_{-mf}KBv = \Pi^{ne}v$ and $M_{i+f} = \Pi^{(n+1)e}v$. Strongness of $FA$ furthermore implies that $F_mAv \neq F_{m-1}Av$ for all $m \leq n$, whence $f = e$.
\end{proof}

\begin{proposition}
Let $A$ be a csa with strong filtration $FA$ extending $FK$ an $f$-step valuation filtration with $F_0K$ a DVR and put $F_0A = B$. Let $M$ be non-empty in $\GBS_F(A)$, then $AM = KM = Av$ with $Av \in \BS(A)$, $v$ any generator of an $A$-module. Then there exists $0 \leq d < f$ and $m \in \mathbb{N}$ such that $M$ is the chain
\begin{equation}\nonumber
\resizebox{0.9 \hsize}{!}{$\begin{array}{ccccccccccccccccccc}
F_dAF_{mf}Kv &\supset& F_{d-1}AF_{mf}Kv &\supset &\ldots& \supset& F_1AF_{mf}Kv& \supset& F_{mf}Kv& \supset &F_{f-1}AF_{(m-1)f}Kv& \supset& \ldots& \supset& F_1AF_{(m-1)f}Kv& \supset& F_{(m-2)}Kv& \supset& \ldots\\
\veq && \veq &&  && \veq && \veq && \veq &&  && \veq && \veq &&\\
M &\supset& M_1 &\supset& \ldots &\supset& M_{d-1} &\supset& M_d &\supset&  M_{d+1} &\supset& \ldots &\supset& M_{d+f -1} &\supset& M_{d+f} &\supset& \ldots
\end{array}$}
\end{equation}
Moreover, the chain $M_d \supset M_{d+f} \supset M_{d+2f} \supset \ldots$ defines an element in $\GBS_F(K)$.
\end{proposition}
 \begin{proof}
 Essentially just modify the proof of \propref{relativecsa}.
 \end{proof}
 
We observe that unlike the case for $f = 1$, strong filtrations $FA$ and $F'A$ extending an $f$-step valuation filtration $FK$ have different $\GBS_F(A) \neq \GBS_{F'}(A)$ if $F_0A \neq F'_0A$. In particular, $\GBS^{F^fK}$ contains as a subset
$$\BS(A) \times A(R,f) \times \mathbb{Z},$$
where $A(R,f)$ denotes the set of all maximal $R$-orders in $A$ with ramification index $f$. When considering all discrete valuations on $K$, for every point of the Riemann surface $R(K)$ we obtain differences corresponding to the maximal orders of ramification index $f$ lying over it. It remains a question whether there exist filtrations $FA$ extending an $f$-step valuation filtration with non-empty $\GBS_F(A)$ and which are not strong.

\section{Tensor product with a field extension}\seclabel{tensor}

Let $A$ be a csa over $K$ with separated, exhaustive filtration $FA$ inducing a filtration $FK$ on $K$. Let $L/K$ be a field extension and choose a filtration $FL$ on $L$ inducing $FK$ on $K$. We define a filtration on $A \ot_K L$ by putting
$$f_q(A \ot_K L) = \sum_{k \leq q} F_kA \ot F_{q-k}L,$$
where the tensor product is taken over $K$. Note that $F_kA$ and $F_{q-k}L$ are not necessarily $K$-vector spaces, but we consider them inside $K$, resp. $L$. 

$$\begin{array}{ccccc}
f_{-2} & f_{-1} & f_0 & f_1 & f_2 \\
\hline
&&&& F_2A \ot F_0L \\
&&& F_1A \ot F_0L & F_1A \ot F_1L\\
&& F_0A \ot F_0L & F_0A \ot F_1L & F_0A \ot F_2L\\
& F_{-1}A \ot F_0L & F_{-1}A \ot F_1L & F_{-1}A \ot F_2L & F_{-1}A \ot F_3L\\
F_{-2}A \ot F_0L & F_{-2}A \ot F_1L & F_{-2}A \ot F_2L & F_{-2}A \ot F_3L & F_{-2}A \ot F_4L
\end{array}$$

Let $\Omega \supset M$ be a (left) $FA$-glider. Inside $\Omega \ot_K L$ we want to define a glider representation for $f(A \ot_K L)$, which we denote by $M \ot L$. We define
\begin{equation}\label{chain}
(M \ot L)_p = \sum_{i \geq p} M_i \ot F_{i-p}L.
\end{equation}

$$\begin{array}{ccc}
0 & 1 & 2 \\
\hline
M_0 \ot F_0L &&\\
M_1 \ot F_1L & M_1 \ot F_0L &\\
M_2 \ot F_2L & M_2 \ot F_1L & M_2 \ot F_0L
\end{array}$$

for $a \ot s \in f_j(A \ot_K L)$ and $m \ot t \in (M \ot L)_i$ with $j \leq i$, we define
\begin{equation} \label{action}
 (a \ot s) \cdot (m \ot t) = (am \ot st).
\end{equation}
\begin{lemma}
The left $A \ot_K L$-module $\Omega \ot_K L$ with chain and partial actions defined by \eqref{chain} and \eqref{action} yields a left $f(A\ot_K L)$-glider representation.
\end{lemma}
\begin{proof}
Let $j \leq i$, then 
\begin{eqnarray*}
\Big(\sum_{k \leq j} F_kA \ot F_{j-k}L\Big)\Big(\sum_{l \geq i} M_l \ot F_{l -i}L\Big) &\subset& \sum_{k \leq j}\sum_{l \geq i} F_kAM_l \ot F_{j-k}L F_{l-i}L\\
&\subset& \sum_{k \leq j} \sum_{l \geq i} M_{l -k} \ot F_{l -k +j-i} \\
&\subset& \sum_{h \geq l-k} M_h \ot F_{h + j -i} = (M \ot L)_{i-j}.
\end{eqnarray*}
\end{proof}

In case we start from a strong filtration $FK$ on $K$, the extended filtration on $L$ satisfies $F_nL = F_nKF_0L$ and the filtration on $A \ot_K L$ becomes $FA \ot F_0L$ and for the glider filtration we obtain the chain
$$M \ot F_0L \supset M_1 \ot F_0L \supset M_2 \ot F_0L \supset \ldots$$

\begin{proposition}\proplabel{map}
When $FK$ is strong, we obtain a map
$$\GBS_F(A) \to \GBS_f(A \ot_K L).$$
\end{proposition}
\begin{proof}
Take $M \in \GBS_F(A)$ and consider an $f(A\ot_K L)$-subglider $N \supset N_1 \supset \ldots$ of $M \ot F_0L \supset M_1 \ot F_0L \supset \ldots.$ By the previous section, we know that $\Omega = Av$ is a simple left ideal, whence so is $Av \ot_K L$. If $N = 0$, then there is nothing to prove. If $N \neq 0$, then it follows that $(A \ot_K L)N = (A \ot_K L)(M \ot_K F_0L)$. In particular, there exist $m \in M, l \in F_0L, a_i \in A, l_i \in F_0L, n_i \in N$ such that
$$m \ot l = \sum_{i=1}^t (a_i \ot l_i)n_i.$$
Because $FA$ is not bounded below, we can find elements $b \in A$ such that $ba_i \in F_0A$ and $bm \in M$. It follows that
$$bm \ot l = \sum_{i=1}^t (ba_i \ot l_i)n_i \in N.$$
In particular, we know that $N$ contains a monomial, say $m \ot l \in M \ot F_0L \cap N$. In fact, for every $i \geq 0$, $N_i$ contains a monomial $m_i \ot l$. Consider now the $FA \ot F_0K$-subglider
$$\begin{array}{cccccccc}
M \ot F_0Kl & \supset & M_1 \ot F_0Kl & \supset & M_2 \ot F_0Kl & \supset & \ldots \\
\cup && \cup && \cup && \\
N \cap (M \ot F_0Kl) & \supset & N_1 \cap (M_1 \ot F_0Kl) & \supset & N_2 \cap (M_2 \ot F_0Kl) & \supset & \ldots
\end{array}$$
Since $FA \ot F_0K \cong FA$ is an isomorphism of filtered rings and since $K$ is central in $A$, the glider $M \ot F_0Kl \supset M_1 \ot F_0Kl \supset \ldots$ is isomorphic to $M \supset M_1 \supset \ldots$, whence is irreducible. If $N \cap (M \ot F_0Kl)$ is trivial of type $T_1$, then there exists $n \geq 0$ such that for all $m \geq 0$
$$N_n \cap (M_n \ot F_0Kl) = N_{n+m} \cap (M_{n+m} \cap F_0Kl) = B(N \cap (M \ot F_0Kl)).$$
In this case, $N_n \cap (M_n \ot F_0Kl) =  B(N \cap (M \ot F_0Kl) \subset B(M \ot F_0Kl)$ would be a non-zero left $A$-module, which contradicts \propref{principal}. Since every $N_i$ contains a monomial of the form $m_i \ot l$, the subglider cannot be trivial of type $T_2$, so it is of type $T_3$ with monotone increasing map $\alpha$. Since we can chose any $l \in F_0L$ above, we can assume $l \in F_0L^\times$. Now, suppose that the monomial $m \ot l \in N$ is such that $m \in M_{s} - M_{s+1}$. Since $FA$ is strong, there exists $c \in F_sA - F_{s-1}A$, whence $(cm \ot l) \in (M \ot F_0L) \cap N$ and not in $M_1 \ot F_0L$. From this it follows that $\alpha = id$. Because $l$ is invertible in $F_0L$, it follows that 
$$M_{m} \ot F_0Ll = M_m \ot F_0L  \subset N_m \subset M_m \ot F_0L,$$
for every $m \geq 0$. 
\end{proof}

\begin{example}
Let $FK$ be the valuation filtration for $F_0K = O_v$ a DVR and let $O_w$ be a valuation extension of $O_v$ with fraction field $L /K$ and consider on $L$ the $w$-adic filtration. If $FA$ extends $FK$, then $f( A \ot_K L)$ extends $FL$, whence also $FK$. Since we have $f_n(A \ot_K L) = F_nA \ot O_w = F_nK(F_0A \ot O_w)$, \propref{map} yields a map
$$\GBS^{FK}(A) \to \GBS^{FK}(A \ot_K L),$$
which sends $M = (F_nKa)_\ast$ to $(F_nKb)_\ast$ if $Aa \ot_K L = (A \ot_K L)b \in \BS(A \ot_K L)$.
By choosing a valuation extension for every element in $R(K)$, we obtain a map
$$ \BS(A) \times \GBS(K) = \GBS^K(A) \to  \GBS^K(A \ot_K L) = \BS(A \ot_K L) \times \GBS(K).$$
This shows that the relative Glider Brauer-Severi variety $\GBS^K(A)$ is a twist of the relative Brauer-Severi variety of matrix algebra $M_n(L)$ for a suitable field extension $L/K$.
\end{example}

\section{The glider Brandt groupoid}\seclabel{Brandt}

Let $K$ be a field with separated, exhaustive filtration $FK$ such that $R = F_0K$ is a Dedekind domain. Consider a central simple $K$-algebra $A$. Classically, one defines the Brandt groupoid as the set of all normal ideals $M$ in $A$ with proper multiplication. Such a normal ideal $M$ is a finitely generated $R$-torsion free $R$-module such that $KM = A$. We will consider finitely generated $R$-torsion free $FK$-gliders 
$$M \supset M_1 \supset M_2 \supset \ldots$$
inside $A$ and such that $KM = A$. Define the following subrings of $A$
$$O_l(M_i) = \{ x \in A~\vline~ xM_i \subset M_i \},$$
$$O_l^{gl}(M) = \{ x \in A~\vline~ xM_i \subset M_i~\forall i \geq 0\}.$$

Suppose that $M$ is generated by $\{(m_i,d_i), 1 \leq i \leq t\}$, i.e. 
$$M_n = \sum_{i=1}^t F_{d_i-n}Km_i$$
for all $n \geq 0$. Moreover, \lemref{fieldjac} and a result from \cite{NVo1} entail that $m_i \in M_{d_i} - M_{d_i +1}$. 

\begin{lemma}
We have the equalities
$$O^{gl}_l(M) = \bigcap_{i=1}^t O_l(M_{d_i}) = \bigcap_{i \geq 0} O_l(M_i).$$
\end{lemma}
\begin{proof}
Suppose that $x \in \bigcap_{i=1}^t O_l(M_{d_i})$ and let $i \geq 0$. We have that
\begin{eqnarray*}
xM_0 &=& x\sum_{k=1}^tF_{d_k - i}Km_k = \sum_{k=1}^tF_{d_k - i}Kxm_k \\
&\subset& \sum_{k=1}^tF_{d_k - i}KM_{d_k} \subset M_i,
\end{eqnarray*}
whence $x \in O_l^{gl}(M)$. The result now follows since we have the inclusions
$$\bigcap_{i=1}^t O_l(M_{d_i}) \subset  O_l^{gl}(M) \subset  \bigcap_{i \geq 0} O_l(M_i) \subset \bigcap_{i=1}^t O_l(M_{d_i}).$$
\end{proof}

We denote $B = B^l = O_l^{gl}(M)$ and define a filtration $FA$ on $A$ by putting $F_0A = B$ and $F_nA = F_nKB$. 
\begin{proposition}
The chain $M \supset M_1 \supset \ldots$ is a left $FA$-glider ideal.
\end{proposition}
\begin{proof}
The previous lemma shows that all $M_i$ are left $B$-modules. For $i \leq j$ we have
$$F_iAM_j = BF_iKM_j \subset BM_{j-i} \subset M_{j-i}.$$
\end{proof}

Analogously, one introduces $O_r^{gl}(M)$, $O_r(M_i)$ and one obtains a subring $B^r$ yielding another filtration $F^rA$ on $A$ such that $M$ becomes a right $F^rA$-glider ideal as well.\\

Suppose now that we have two finitely generated $R$-torsion free $FK$-gliders in $A$, say $M$ and $N$. For $i \geq 0$, define
$$(M\cdot N)_i = \{ \sum_j^{'}m_jn_j ~|~ \gr(m_j) + \gr(n_j) \geq i\}.$$
It is clear that $(M\cdot N)_i$ is an $R$-module. 
\begin{lemma}\lemlabel{product}
We have the equality
$$(M\cdot N)_i = \sum_{k=0}^i M_kN_{i-k}.$$
\end{lemma}
\begin{proof}
Let $x \in (M \cdot N)_i$ and suppose that $m_in_i$ appears in the expression of $x$ as a finite sum. If $\gr(m_i) \geq i$, then $m_in_i \in M_kN$. If $0 \leq \gr(m_i) < i$ then from $\gr(n_i) \geq i - \gr(m_i)$ it follows that $n_i \in N_{i - \gr(m_i)}$. The other inclusion is trivial.
\end{proof}
Suppose that $M$ and $N$ are generated by $\{(m_i,d_i),~1 \leq i \leq t\}$ and $\{ (n_j,e_j),~ 1 \leq j \leq s\}$ respectively. The previously lemma then entails that $M\cdot N$ is finitely generated by $\{(m_in_j, d_i+e_j), 1 \leq i \leq t, 1 \leq j \leq s\}$. Indeed,
\begin{eqnarray*}
(M \cdot N)_l &=& \sum_{k = 0}^l M_kN_{l-k} \\
&=& \sum_{k=0}^l \sum_{i=1}^t \sum_{j = 1}^s F_{d_i-k}KF_{e_j - l +k}K m_in_j\\
&\subset&  \sum_{i=1}^t \sum_{j = 1}^s F_{d_i+e_j -l}Km_in_j \subset (M\cdot N)_l.
\end{eqnarray*}

Our next goal is to define the inverse glider 
$$M^{-1} \supset (M^{-1})_1 \supset (M^{-1})_2 \supset \ldots$$
To this extent, we put
$$M^{-1} = \{ x \in A ~\vline~ MxM \subset M\},$$
as in the classical situation of normal ideals. For $i > 0$ we then define
$$(M^{-1})_i = \{ x \in A \vline~ MxM \subset M_i\}.$$
\begin{proposition}
$M^{-1}$ together with the chain defined by the $(M^{-1})_i$ is a finitely generated $R$-torsion free $FK$-glider.
\end{proposition}
\begin{proof}
Let $i \leq j, x \in (M^{-1})_j$, then $MF_iKxM \subset F_iKM_j \subset M_{j-i}$ shows that
$F_iK(M^{-1})_j \subset (M^{-1})_{j-i}$ and $M^{-1}$ is indeed an $FK$-glider. $R$-torsion freeness is obvious. Next, we show that $M^{-1}$ is a finitely generated $R$-module. To this extent, let $a \in A$. The $R$-module $MaM$ is finitely generated and since $KM = A$ it follows that there exists $r \in R$ such that $MraM = rMaM \subset M$. Hence $ra \in M^{-1}$ and it follows that $KM^{-1} = A$. We have an embedding as $R$-modules
$$MM^{-1} \hookmapright{} \Hom_R(M,M),~ w \mapsto (m \mapsto wm).$$
Indeed, suppose that $w, w' \in MM^{-1}$ define the same morphism, then $wM = w'M$ element-wise and since $MK = A$, also $wA = w'A$ element-wise. In particular, $w = w'$. By choosing generators $u_1, \ldots, u_n$ for $M$ as $R$-module, we obtain a surjective morphism $M_n(R) \twoheadrightarrow \Hom_R(M,M)$. Since $M_n(R)$ is a finitely generated $R$-module, so is $\Hom_R(M,M)$. Because $R$ is Noetherian, it follows that $MM^{-1}$ is also finitely generated. The morphism
$$M^{-1} \to \Hom_R(M,MM^{-1}),~ w \mapsto (m \mapsto mw)$$
is again injective. The $R$-module $\Hom_R(M,MM^{-1})$ is the image of $M_{d,n}(R)$ by choosing generators $y_1,\ldots, y_d$ in $MM^{-1}$ and we arrive at $M^{-1}$ being finitely generated as $R$-module. Suppose now that
$M^{-1} = \sum_{i=1}^t Rx_i$, then we define $j_i \geq 0$ to be maximal such that $Mx_iM \subset M_{j_i}$. It follows that for all $n \geq 0$, $\sum_{i = 1}^t F_{j_i -n}Kx_i \subset (M^{-1})_n$. By enlarging one of the indices $F_{j_i - n + d}$ the sum on the left sits inside $(M^{-1})_{n-d}$ which shows that the above inclusion is in fact an equality. This shows that $M^{-1}$ as an $FK$-glider is generated by $\{(x_i,j_i),~ 1 \leq i \leq t\}$.
\end{proof}

By definition, we have
$$((M^{-1})^{-1})_i = \{ x \in A ~\vline~ M^{-1}xM^{-1} \subset (M^{-1})_i\}.$$

\begin{proposition}\proplabel{inverse}
If $B$ is a maximal order, then $M = (M^{-1})^{-1}$ as $FK$-gliders.
\end{proposition}
\begin{proof}
Let $x$ be such that $M^{-1}xM^{-1} \subset (M^{-1})_i$. This means that 
$$MM^{-1}xM^{-1}M \subset M_i.$$
Since $B$ is a maximal order and $M$ is a normal ideal in $A$, the classical theory yields that $MM^{-1} = B$ and $M^{-1}M = B^r$. In particular, $1 \in B \cap B^r$, which shows that $x \in M_i$. Conversely, let $x \in M_i$. We have to show that $MM^{-1}xM^{-1}M \subset M_i$. This follows since $M_i$ is a $B-B^r$-bimodule.
\end{proof}

From now on, we assume $B$ is a maximal order.

\begin{corollary}
$M\cdot M^{-1}$ is a left $FA$-glider ideal.
\end{corollary}
\begin{proof}
From the deductions made after \lemref{product}, we know that $M \cdot M^{-1}$ is a finitely generated $FK$-glider. Since $F_nA = F_nKB$ and all $M_i$ are left $B$-modules, $M\cdot M^{-1}$ is indeed a left $FA$-glider. 
\end{proof}

\begin{remark}\remlabel{right} 
The same results holds on the right, that is with regards to $B^r$ and $M^{-1}\cdot M$.
\end{remark}

Let us recall the definition of a groupoid $G$ from \cite{Rein} as being a collection of elements, certain of whose products are defined and lie in $G$, such that
\begin{enumerate}
\item For each $a_{ij} \in G$ there exist unique elements $e_i, e_j \in G$ such that $e_ia_{ij} = a_{ij} = a_{ij}e_j$, where all indicated products are defined. Further, $e_ie_i = e_i, e_je_j = e_j$. We call $e_i$ the left unit of $a_{ij}$ and $e_j$ the right unit of $a_{ij}$;
\item $a_{ij}b_{kl}$ is defined if and only if $j = k$, that is, if and only if the right unit of $a_{ij}$ equals the left unit of $b_{kl}$;
\item If $ab$ and $bc$ are defined, so are $(ab)c$ and $a(bc)$ and these are equal;
\item For each $a_{ij} \in G$, there exists an $a^{-1}_{ij} \in G$ with left unit $e_j$, right unit $e_i$, such that 
$$a_{ij}\cdot a_{ij}^{-1} = e_i, \quad a^{-1}_{ij}\cdot a_{ij} = e_j.$$
\item Given any pair of units $e,e' \in G$, there is an element $a_{ij} \in G$ with left unit $e$, right unit $e'$.
\end{enumerate}

We call a finitely generated $R$-torsion free $FK$-glider $M$ inside $A$ such that $KM = A$ and such that both $B^l = B^l(M)$ and $B^r = B^r(M)$ are $R$-orders a normal $FK$-glider ideal in $A$. Let $M,N$ be two normal $FK$-glider ideals in $A$.  We have to be careful, however, with the notion of left and right units, as there can exist multiple elements $e$ satisfying $e \cdot M = M$. Indeed,

\begin{lemma}\lemlabel{helpresults}
Let $FA$ be a filtration extending $FK$ with $F_0A = B$ and negative part $F^{-}B$. Let $M \supset M_1 \supset \ldots$ be a left $FK$-glider.
\begin{enumerate}
\item If $M$ is a left $FA$-glider, then $F^{-}B \cdot M = M$,
\item  If $F^{-}B \cdot M = M$ then $M$ is a left filtered $F^{-}B$-module for the chain $F_{-n}M = M_n$,
\item $M$ is an idempotent, i.e. $M \cdot M = M$ if and only if  
$$M \supset M_1 \supset M_2 \supset \ldots$$
defines a negative algebra filtration on the ring $M \subset A$. 
\end{enumerate}
\end{lemma}
\begin{proof}
(1) and (3) are straightforward. We prove (2). Let $d, n \geq 0$. Then
$$F_{-d}BM_n = F_{-d}BM_{n-d+d} \subset (F^{-}B \cdot M)_{n-d} = M_{n-d}.$$
\end{proof}

For two idempotents $F^{-}B$ and $F^{-}C$ we put $F^{-}B \leq F^{-}C$ if and only if $F^{-}B \cdot F^{-}C = F^{-}C$. 

\begin{proposition}\proplabel{unit}
Let $M$ be a normal $FK$-glider ideal in $A$. The set 
$$\{ e {\rm~idempotent~} \vline~ e \cdot M = M\}$$
has a unique maximal element $E^l  = E^l(M)$, which we call the left unit of $M$. The same result holds for right multiplication, leading to a right unit $E^r = E^r(M)$.
\end{proposition}
\begin{proof}
The glider ideal $M$ yields a subring $B = B^l(M)$ and filtration $FA$ on $A$ such that $M$ is a left $FA$-glider. The theory of glider modulizers, see \cite{NVo2}, shows that there exists a subring $B^\ast \supset B$ with chain $F^{-,\ast}B^\ast$ which yields a negative ring filtration and such that $M$ with filtration $F_{-n}M = M_n$ is a left filtered $F^{-,\ast}B^\ast$-module. We recall that the negative part of the chain is defined by 
$$F_{-d}^\ast B = \{ x \in B ~\vline~ xM_{n-d} \subset M_n {\rm~for~} n \in \mathbb{N} {\rm~such~that~} n-d \geq 0\},$$
for $d \geq 0$. This entails that
$$F^{-,\ast}B^\ast \cdot M = M.$$ 
In fact, by definition of $B$, $B^\ast$ actually equals $B$. Now suppose that $M$ is also a left $fA$-glider for some exhaustive, separated filtration $fA$ on $A$. Let $x \in f_{-d}A$, then for all $n \geq d$ it holds
$$xM_{n-d} \subset M_n,$$
which entails that $x \in F^{-d,\ast}B$. This shows that $f^{-}A\cdot F^{-,\ast}B = F^{-,\ast}B$ and $E = F^{-,\ast}B$.
\end{proof}

The previous proposition leads to calling the multiplication $M\cdot N$ proper if $E^r(M) = E^l(N)$. We will show that the collection of normal $FK$-glider ideals in $A$ with proper multiplication is in fact a groupoid.

\begin{proposition}\proplabel{u}
We have the equalities
$$E^l(M) = F^{-,\ast}B^l = M\cdot M^{-1},$$
and
$$E^r(M) = F^{-,\ast}B^r = M^{-1} \cdot M.$$
\end{proposition}
\begin{proof}
Let $x \in F_{-d}^*A$, i.e. $xM_{n-d} \subset M_n$ for all $n \geq d$. Then
$$x \in xB = xMM^{-1} = xM_{d-d}M^{-1} \subset M_dM^{-1} \subset (M\cdot M^{-1})_d.$$
This shows that $F^{-,\ast}B \subset M \cdot M^{-1}$. \lemref{helpresults} then entails that 
$$M \cdot M^{-1} \subset (M \cdot M^{-1})\cdot (M \cdot M^{-1}).$$ 
Let $i \geq 0$, then we have
\begin{eqnarray*}
((M \cdot M^{-1}) \cdot (M \cdot M^{-1}))_i &=& \sum_{k=0}^i\sum_{s = 0}^k \sum_{t = 0}^{i-k} M_s(M^{-1})_{k-s}M_t(M^{-1})_{i-k-t}\\
&\subset&  \sum_{k=0}^i\sum_{s = 0}^k \sum_{t = 0}^{i-k} M_sB^r(M^{-1})_{i-k-t} \\
&\subset&  \sum_{k=0}^i\sum_{s = 0}^k \sum_{t +k= 0}^{i}M_s(M^{-1})_{i-k-t} \\
&\subset& \sum_{s = 0}^k M_s (M^{-1})_{i-s} = (M \cdot M^{-1})_i.
\end{eqnarray*}
 It follows that $M\cdot M^{-1}$ is an idempotent, hence equals $E^l(M)$ by \propref{unit}. The proof for $E^r(M)$ is analogous.
\end{proof}

\begin{proposition}
Let $FK$ be a separated, exhaustive filtration on a field $K$ and let $A$ be a central simple $K$-algebra. The collection of all normal $FK$-glider ideals in $A$ with proper multiplication $M \cdot N$ and units being the idempotent elements forms a groupoid.
\end{proposition}
\begin{proof}
Let $M, N$ and $V$ be normal $FK$-glider ideals such that $M\cdot N$ and $N \cdot V$ are defined. For $i \geq 0$ we have by definition
\begin{eqnarray*}
((M\cdot N) \cdot V)_i &=& \sum_{k=0}^i \sum_{j=0}^k M_jN_{k-j}V_{i-k} \\
&=& \sum_{j=0}^i\sum_{k=j}^i M_jN_{k-j}V_{i-k}\\
&=& \sum_{j=0}^i \sum_{l=0}^{i-j}M_jN_lV_{i-j-l}\\
&=& (M \cdot (N \cdot V))_i.
\end{eqnarray*}
Properties (1) and (4) follow from \propref{u}, for property (5), the element $F^{-}B\cdot F^{-}C$ with $F^{-}B,F^{-}C$ being two units, i.e. negative algebra filtrations on subrings $B$ and $C$ of $A$, does the trick.

\end{proof}

\begin{example}
When $FK$ is strong, it follows from $M$ being finitely generated, that $M_i = F_{-i}KM$ for all $i \geq 0$. This shows that $B = O_l(M) = O_l(M_i)$ for all $i \geq i$. One then shows that $M_i (M^{-1})_j \subset F_{-i-j}A$. Since the filtration $FA$ is also strong, one subsequently shows that 
$$(M \cdot M^{-1})_i = F_{-i}A.$$
Thus, when $FK$ is strong and $B$ is a maximal order, then $M \cdot M^{-1}$ is the left $FA$-glider
$$B \supset F_{-1}A \supset F_{-2}A \supset \ldots$$
and $F^-B = F^{-,\ast}B = M\cdot M^{-1}$
and similar for $B^r$.
\end{example}

\section{Higher rank valuations}\seclabel{Higher}

Up to now, we have restricted to ring filtrations filtered over the integers $\mathbb{Z}$. The theory for gliders can more generally be introduced for general totally ordered abelian groups $\Gamma$. One result we will need, is that if $M$ is an irreducibe $FR$-glider where $FR$ is $\Gamma$-filtered, then for any $\gamma \in \Pi$, where $\Pi$ is a positive cone of $\Gamma$, the glider $M_\gamma$ given by $(M_\gamma)_\delta = M_{\gamma + \delta}, \delta \in \Pi$, remains irreducible.\\

In this section we consider $\Gamma$-filtered rings where $\Gamma = \mathbb{Z}^n$ equipped with the lexicographical order, that is
$$(a_1,\ldots,a_n) < (b_1,\ldots,b_n) \Leftrightarrow \left\{ \begin{array}{l} {\rm~the~first~non-zero~element~of}\\
b_1-a_1,\ldots, b_n - a_n {\rm~is~positive.} \end{array}\right.$$

In fact, we will restrict to $n = 2$. Let us recall from \cite{Mat} how one obtains higher rank (commutative) valuation rings. To this extent, let $K$ be a field with valuation, and call $R$ its valuation ring. If $M$ denotes the unique maximal ideal of $R$, then we can equip the residue field $K' = R/M$ again with a valuation. If we denote its valuation ring with 
$$ \ov{R}' \subset K',$$
then the subset of $R$ given by
$$R' := \{ x \in R~\vline~ x + M \in \ov{R}'\},$$
is a valuation ring. The value groups of $R, R'$ and $\ov{R}'$ are related by a short exact sequence
$$0 \to \Gamma_{\ov{R}'} \to \Gamma_{R'} \to \Gamma_{R} \to 0.$$
We call $R'$ the composite of the valuation rings $R$ and $\ov{R}'$. 

\begin{example}
Let $K = k((y))((x))$ be the fraction field of $R = k((y))[[x]]$. The $x$-adic valuation yields a discrete valuation on $K$ with valuation ring $R$ and value group $\mathbb{Z}$. On the residue field $K' = k((y))$ we put the $y$-adic valuation with valuation ring $k[[y]]$. The composite is the valuation ring $R' = k[[y]] + xk((y))[[x]]$ and the short exact sequence
$$0 \to \mathbb{Z} \to \Gamma_{R'} \to \mathbb{Z} \to 0$$
splits, so that $\Gamma_{R'} = \mathbb{Z}^2$. One checks that the total ordering is given by the lexicographical ordering. The negative part of the associated valuation filtration on $R' = F_{(0,0)}K$ is given by
$$F_{(-i,-j)}K = x^i(y^jk[[y]] + xk((y))[[x]]).$$
\end{example}

From now on, all filtrations $FK$ are assumed to be $\mathbb{Z}^2$-filtrations with lexicographical ordering. In order to generalize \theref{GBSfield} for fields $K$ equipped with $\mathbb{Z}^2$-filtrations, we first observe that the additive chain 
$$ f^hK : \quad \ldots \subset \bigcup_{n \in \mathbb{Z}} F_{(-1,n)}K \subset  \bigcup_{n \in \mathbb{Z}} F_{(0,n)}K  \subset  \bigcup_{n \in \mathbb{Z}} F_{(1,n)}K  \subset \ldots$$
defines a separated, exhaustive $\mathbb{Z}$-filtration on $K$. We also have a filtration on $R = \bigcup_{n \in \mathbb{Z}} F_{(0,n)}$ given by
$$f^vR: \quad \ldots F_{(0,-2)}K \subset  F_{(0,-1)}K \subset  F_{(0,0)}K \subset  F_{(0,1)}K \subset  F_{(0,2)}K \subset \ldots$$
Let $M$ be an $FK$-glider. We will depict this by
$$\begin{array}{ccccccc}
\vdots && \vdots && \vdots && \ddots \\
\cap && \cap && \cap && \\
M_{(0,2)} & \supset & M_{(1,2)} & \supset &M_{(2,2)} & \supset & \cdots \\
\cap && \cap && \cap && \\
M_{(0,1)} & \supset & M_{(1,1)} & \supset &M_{(2,1)} & \supset &\cdots \\
\cap && \cap && \cap && \\
M_{(0,0)} & \supset & M_{(1,0)} & \supset &M_{(2,0)} & \supset &\cdots 
\end{array}$$
 For $m \geq 0$, the chain 
$$M_{(m,\ast)}: \quad M_{(m,0)} \supset M_{(m,1)} \supset M_{(m,2)} \supset \ldots$$
is an $f^vR$-glider, whence its body $B(M_{(m,\ast)})$ is an $R$-module.
\begin{lemma}
Let $M \in \GBS_F(K)$, then the chain 
$$B^v(M): \quad B(M_{(0,\ast)}) \supset B(M_{(1,\ast)}) \supset B(M_{(2,\ast)}) \supset \ldots$$
is an irreducible $f^hK$-glider, i.e. $B^v(M) \in \GBS_{f^hK}(K)$.
\end{lemma}
\begin{proof}
That $B^v(M)$ is an $f^hK$-glider is obvious. Suppose that $T \supset T_1 \supset \ldots$ is a subglider. There is an associated $FK$-subglider of $M$, given by
$$\begin{array}{ccccccc}
\vdots && \vdots && \vdots && \ddots \\
\cap && \cap && \cap && \\
T & \supset & T_1 & \supset &T_2 & \supset & \cdots \\
\cap && \cap && \cap && \\
T & \supset & T_1 & \supset &T_2 & \supset &\cdots \\
\cap && \cap && \cap && \\
T & \supset & T_1 & \supset &T_2 & \supset &\cdots 
\end{array}$$
so it must be trivial. It is clear that triviality of type $T_1$ or $T_2$ easily yield that $T \supset T_1 \supset \ldots$ is trivial of the same type. So suppose that $T \supset T_1 \supset \ldots$ is not trivial of type $T_1$ or $T_2$. It follows that there exists a monotone increasing map $\alpha: \mathbb{N}^2 \to \mathbb{N}^2$. We define a map $\beta: \mathbb{N} \to \mathbb{N}$ by
$$\beta(n) := \pi_1( \sup_{m} \alpha(n,m)),$$
where $\pi_1: \mathbb{N}^2 \to \mathbb{N}$ is the projection on the first component. Observe that our assumption that the subglider is not of type $T_1$ or $T_2$ ensures that the supremum is indeed finite. One checks that $\beta$ is monotone increasing and satisfies $T_n = B(M_{(\beta(n),\ast)})$. This shows that $T \supset T_1 \supset \ldots$ is trivial of type $T_3$. 
\end{proof}

Since we have that $B^v(M)_0 \subset f^h_0K$, we can invoke \propref{associatedfil} to refine the negative part of $f^hK$ to obtain a strong $e$-step filtration $f^sK$ such that $B^v(M) \in \GBS_{f^s}(K)$. We know that the negative part of $f^hK$ is a trivial $f^sK$-subfragment of type $T_3$ of the negative part of $f^sK$. Suppose that $\alpha: \mathbb{N} \to \mathbb{N}$ yields such a relation, i.e. $f^v_{-n}K = f^s_{-\alpha(n)}K$. We refine the negative part of $FK$, i.e. when the first component is negative by putting
$$F'_{(-\alpha(n),m)}K = F_{(-n,m)}K,$$
and 
$$F'_{(-i, m)}K = f^s_{-i}K \cap F_{(-n-1,m)}K$$
if $\alpha(n) < i < \alpha(n+1)$ for some $n$.
Since we only altered the horizontal direction, $M \in \GBS_{F'K}(K)$, so we may replace $F'K$ by $FK$ without harm. Suppose now that $e > 1$. In particular, we have that $f^s_{-1}K = f^s_{-2}K$. It follows that
$$f_{-1}K \supset \cap_{n \in \mathbb{Z}} F_{(-1,n)}K \supset f_{-2},$$
whence $f_{-1}K = F_{(-1,n)} = f_{-2}K$ for all $n \in \mathbb{Z}$. Next, since $f^s_0K.f^s_{-1}K = f^s_{-1}K$ we obtain 
$$F_{(0,n)}KF_{(-1,0)}K = F_{(-1,0)}K,$$
for all $n \in \mathbb{Z}$ and since we are working in a field, it follows that $f^s_0K = F_{(0,n)}K$ for all $n \in \mathbb{Z}$. After performing similar reasonings, one deduces that $FK$ is trivially filtered in the vertical direction, and so $FK$ is essentially a $\mathbb{Z}$-filtration, which we can exclude. Hence $e = 1$ and we may assume that $f^hK = f^sK$. 

\propref{fieldstrong} entails that $f^s_0K = f^h_0K = R$ is a DVR and $f^hK$ is the associated valuation filtration. One checks that 
$$gK': \quad \ldots \subset \frac{F_{(0,-2)}K}{F_{(-1,-2)}K} \subset \frac{F_{(0,-1)}K}{F_{(-1,-1)}K} \subset \frac{F_{(0,0)}K}{F_{(-1,0)}K} \subset \frac{F_{(0,1)}K}{F_{(-1,1)}K} \subset \frac{F_{(0,2)}K}{F_{(-1,2)}K} \subset \ldots$$ 
defines a separated, exhaustive filtration on the residue field $K' = f^h_0K/f_{-1}^hK$. 

\begin{lemma}
The descending chain
$$M^{res}: \quad \frac{M_{(0,0)}}{M_{(1,0)}} \supset \frac{M_{(0,1)}}{M_{(1,1)}} \supset \frac{M_{(0,2)}}{M_{(1,2)}} \supset \ldots$$
is an irreducible $gK'$-glider.
\end{lemma}
\begin{proof}
Suppose that $\ov{N} \supset \ov{N}_1 \supset \ov{N}_2 \supset \ldots$ is a non-trivial $gK'$-subglider. Without loss of generalization, we may assume that $\ov{N}$ corresponds to an $F_{(0,0})K$-submodule 
$$M_{(0,0)} \supsetneq N \supsetneq M_{(1,0)}.$$
However, $F_{(-i,-j)}KN$ would define a non-trivial $FK$-subglider of $M$. 
\end{proof}

Observe that we do not know whether $M^{res} \in \GBS_{gK'}(K')$ as we do not know whether $M_{(1,0)} \subset f^h_{-1}K$. However, we can consider for any $s \geq 0$, the descending chain
$$M^{res}_s: \quad \frac{M_{(s,0)}}{M_{(s+1,0)}} \supset \frac{M_{(s,1)}}{M_{(s+1,1)}} \supset \frac{M_{(s,2)}}{M_{(s+1,2)}} \supset \ldots$$
One shows analogously that they are all irreducible $gK'$-gliders. Hence, they must be all isomorphic since they are contained in one another. Choosing $s$ large enough such that $M_{(s+1,1)} \subset f_{-1}^hK$ then shows that $\GBS_{gK'}(K') \neq \emptyset$. Hence, we can again invoke \propref{fieldstrong} to deduce that $g_0K'$ is a DVR and $g'K$ is the strong $f$-step valuation filtration for some $f \geq 1$.\\

We denote by $\GBS_{\mathbb{Z}^2}(K)$ the union of all $\GBS_F(K)$ where $FK$ is a $\mathbb{Z}^2$-filtration with lexicographical ordering. Let $K/k$ be a field extension, then $R(K/k,2)$ denotes the Riemann surface of all rank 2 valuations containing $k$, i.e. which are trivial on $k$. We have the generalization of \theref{GBSfield}.

\begin{theorem}\thelabel{GBS2}
Let $K/k$ be a field of transcendence degree $\tdeg_k(K) \geq 2$. We have an isomorphism as sets
$$\GBS_{\mathbb{Z}^2}(K) =R(K/k,2) \times \mathbb{Z}^2.$$
\end{theorem}
\begin{proof}
We showed that if $M \in \GBS_F(K)$, then $M \in \GBS_{F^v}(K)$ where $F^vK$ is the valuation filtration of a valuation of rank 2. From the structure of irreducible gliders for $\mathbb{Z}$-filtrations on fields we know that there exists $m \in \mathbb{Z}$ such that 
$$\bigcap_{i} M_{(0,i)} = \bigcup_{j \in \mathbb{Z}} F_{(m,j)}^vK \supset M_{(1,0)}.$$
Since $M_{(1,0)}$ is an $F^v_{(0,0)}K$-module, there exists $n \in \mathbb{Z}$ such that 
$$M_{(1,0)} = F_{(m,n)}^vK.$$
It then follows that 
$$M_{j,i} = F_{m-1-j,n-i}^vK$$
for all $i,j \geq 0$. 
\end{proof}

\end{document}